\documentclass[10pt]{amsart}
\usepackage{amsmath}
\usepackage{amssymb,amscd}
\usepackage{graphicx}
\usepackage{mathdots}
\usepackage{mathrsfs}       
\usepackage{graphicx}
\usepackage{bbm} 
\usepackage{microtype} 

\usepackage{colonequals} 
\usepackage{mathtools}

\usepackage{tikz,tkz-euclide}
\usepackage{tikz-cd,calc}
\usetikzlibrary{arrows}
\usepackage{tikz-3dplot}

\usepackage{color}\definecolor{darkblue}{rgb}{0,0.1,.5}
\usepackage[colorlinks=true,linkcolor=darkblue, urlcolor=darkblue, citecolor=darkblue]{hyperref}
\usepackage{cleveref}

\theoremstyle{plain}
\newtheorem{theorem}{Theorem}[section]
\newtheorem{lemma}[theorem]{Lemma}
\newtheorem{proposition}[theorem]{Proposition}

\newtheorem{conjecture}[theorem]{Conjecture}

\theoremstyle{definition}
\newtheorem{definition}[theorem]{Definition}
\newtheorem{example}[theorem]{Example}

\theoremstyle{remark}
\newtheorem*{remark}{Remark}

\numberwithin{equation}{section}

\def \begineq{\begin{equation}}
\def \endeq{\end{equation}}

\def \bb{\mathbb}

\def \RR{{\bb{R}}}

\def \D{{\bb{D}}}
\def \T{{\bb{T}}}

\def\Bier{\mathrm{Bier}}
\def\MF{\mathrm{MF}}
\def\M{\mathrm{M}}
\def\vc{\mathrm{vc}}

\def\cat{\mathrm{cat}}
\def\sk{\mathrm{sk}}

\def\zk{\mathcal Z_K}
\def\rk{\mathcal R_K}

\def\N{\mathbb N}
\def\Z{\mathbb Z}

\def\R{\mathbb R}
\def\C{\mathbb C}

\DeclareMathAlphabet{\mathbbmsl}{U}{bbm}{m}{sl}

\DeclareMathOperator{\link}{Link}
\DeclareMathOperator{\del}{Del}
\DeclareMathOperator{\cone}{Cone}

\makeatletter
\@namedef{subjclassname@2020}{%
  \textup{2020} Mathematics Subject Classification}
\makeatother

\title[Chromatic numbers, Buchstaber numbers and chordality of Bier spheres]{Chromatic numbers, Buchstaber numbers\\ and chordality of Bier spheres}

\author[Limonchenko]{Ivan Limonchenko}
\address{Mathematical Institute of the Serbian Academy of Sciences
and Arts (SASA), Belgrade, Serbia}
\email{ivan.limoncenko@turing.mi.sanu.ac.rs}

\author[Vavpeti\v{c}]{Ale\v{s} Vavpeti\v{c}}
\address{Faculty of Mathematics and Physics, University of Ljubljana, Slovenia\newline\indent
Institute of Mathematics, Physics and Mechanics, Ljubljana, Slovenia}
\email{ales.vavpetic@fmf.uni-lj.si}



\begin{document}

\begin{abstract}
We describe all the Bier spheres of dimension $d$ with chromatic number equal to $d+1$ and prove that all other $d$-dimensional Bier spheres have chromatic number equal to $d+2$, for any integer $d\geq 0$. Then we prove a general formula for complex and mod $p$ Buchstaber numbers of a Bier sphere $\Bier(K)$, for each prime $p\in\N$ in terms of the $f$-vector of the underlying simplicial complex $K$. Finally, we classify all chordal Bier spheres and obtain their canonical realizations as boundaries of stacked polytopes. 
\end{abstract}

\maketitle

\section{Introduction}

This paper is a part of a program initiated in~\cite{LS}, with the goal of constructing smooth moment-angle manifolds and toric manifolds over Bier spheres and applying them to solve a few important problems in algebraic and equivariant topology.
Moment-angle manifolds and (quasi)toric manifolds
were introduced in~\cite{DJ} and have become the main objects of study in toric topology~\cite{TT}. It is a modern area at the crossroads of algebraic and equivariant topology, where combinatorial properties of an abstract simplicial complex $K$ on $[m]:=\{1,2,\ldots,m\}$ are studied in their relationship with the topological properties of the corresponding moment-angle complex $\zk\subset\C^m$, its real counterpart $\rk\subset\R^m$, as well as their orbit spaces with respect to a freely acting subtorus $H$ in $\T^m$ and $\Z_2^m$, respectively. The latter class of partial quotients includes (quasi)toric manifolds and their real counterparts as the subclass, where the toric subgroup $H$ has a maximal possible rank. 

Quasitoric manifolds found numerous applications in algebraic and equivariant topology, with probably the most remarkable one being that they yield geometric representatives in unitary~\cite{BPR} and special unitary~\cite{CLP} bordism ring classes. However, quasitoric manifolds are not necessarily toric and therefore do not solve the Hirzebruch problem in algebraic topology: find a nonsingular connected algebraic variety representing a given unitary bordism class. One of our goals in the above mentioned program is to construct nonsingular complete toric varieties (toric manifolds) over Bier spheres and describe all the unitary bordism classes in which there exists such a toric manifold. Other goals will also be described in this section of the paper. 

Recall that in the unpublished note~\cite{Bier}, any (abstract) simplicial complex $K$ on $[m]$ with $m\geq 2$, different from the whole simplex $\Delta_{[m]}$ on $[m]$, was associated with a simplicial complex $\Bier(K)$ defined as the deleted join of $K$ and its Alexander dual $K^\vee$. Moreover, it was shown in~\cite{Bier} that $\Bier(K)$ is an $(m-2)$-dimensional PL-sphere with a number of vertices varying between $m$ and $2m$. Since that time, this beautiful and simple construction has been studied intensively and found important applications in topological combinatorics~\cite{Matousek},  polytope theory~\cite{Zivaljevic21}, optimization theory~\cite{Zivaljevic19}, game theory~\cite{TZJ}, combinatorial commutative algebra~\cite{HK, Mu} and toric topology~\cite{LS, LZ}. New proofs of the fact that it yields a sphere were also obtained, see~\cite{dL, Matousek, Mu}.

Combinatorial and geometrical methods were used to classify all Bier spheres of dimension less than three~\cite{LS, LZ}, all flag Bier and Murai spheres~\cite{HK, LZ} and to show that Bier spheres and their generalizations are shellable and edge decomposable~\cite{BPSZ05, Mu} as well as that asymptotically the vast majority of them are non-polytopal~\cite{Matousek}, although no particular examples of a non-polytopal Bier sphere have been found so far. Furthermore, methods of homological algebra made it possible to compute certain multigraded Betti numbers of Bier spheres and their generalizations~\cite{HK, Mu}. Finally, the application of topological combinatorics methods showed that all Bier spheres are starshaped~\cite{Zivaljevic19}, and methods of toric topology yielded the classification of all minimally non-Golod Bier spheres~\cite{LZ}. 

The structure of the paper is as follows. In Section 2, we recall the key definitions and introduce basic examples related to the notions of a simplicial complex, simple graph, convex polytope, and Bier sphere.

The first combinatorial invariant of a Bier sphere studied in this paper is the chromatic number discussed in Section 3. By definition, this is the least number of colors needed for a proper coloring of the vertex set of a simplicial complex; that is, the vertices of each simplex are colored in pairwise different colors. The chromatic number is one of the most important classical combinatorial invariants of a simple graph as well as of a simple polytope. By definition, the chromatic number $\chi(P)$ of an $n$-dimensional convex simple polytope $P$ equals that of its nerve complex $\partial P^*$ and hence $m\geq \chi(P)\geq n$, where $m$ is the number of facets of $P$. A combinatorial description of $n$-colorable simple $n$-polytopes was obtained in~\cite{Joswig}.  Computation of chromatic numbers for families of simple polytopes is a classical open problem in graph theory, polytope theory and applications, see~\cite{BIP}.

Recall that a triangulated sphere is called polytopal if it is simplicially isomorphic to the boundary of a convex simplicial polytope. We prove that for any simplicial complex $K\neq\Delta_{[m]}$ on $[m]$ with $m\geq 2$ the chromatic number $\chi(\Bier(K))$ equals either $m-1$, or $m$. More precisely, the following statement holds, see \Cref{BierChromaticTheorem}.

\begin{theorem}
Let $K\neq\Delta_{[m]}$ be a simplicial complex on $[m]$ with $m\geq 2$. Then 
\begin{enumerate}
\item[(a)] $m-1\leq \chi(\Bier(K))\leq m$;
\item[(b)] $\chi(\Bier(K))=m-1$ if and only if $\Bier(K)$ is polytopal and the next equivalent conditions hold:
\begin{itemize}
\item $K\in\{C^{m-3}(\Gamma_4)$, $C^{m-3}(\Gamma_4^{\vee})$, $C^{m-3}(G_4)$, $C^{m-3}(\Gamma_6)\}$;
\item $P_K\in\{I^{m-1}, I^{m-3}\times P_6\}$.
\end{itemize}
\end{enumerate}
\end{theorem}

Here, we use the notation $C^{\ell}(\mathcal K)$ for the $\ell$ times iterated cone over a simplicial complex $\mathcal K$ and the corresponding 4 graphs are described in Figure~\ref{fig:1complexes}. We also denote by $P_K$ a simple polytope such that for a polytopal Bier sphere $\Bier(K)$ one has a combinatorial equivalence $\Bier(K)=\partial P_K^*$ and use the notations $I^n$, $n\geq 0$ for the standard $n$-dimensional cube and $P_m$, $m\geq 3$ for an $m$-gon. A combinatorial classification of simple $n$-polytopes $P$ with $\chi(P)=n+1$ remains unknown in polytope theory. Our theorem implies that every Bier polytope $P_K$ with $K$ different from each of the complexes described explicitly in~\Cref{BierChromaticTheorem} belongs to this class and, in view of the result of~\cite{Joswig}, it has a 2-dimensional face with an odd number of vertices. 

The second combinatorial invariant of a Bier sphere studied in this paper is the Buchstaber number discussed in Section 4. By definition, the (complex) Buchstaber number of a simplicial complex $K$ with $m$ vertices is the maximal integer $r$ such that there exists a characteristic map $\Lambda\colon [m]\to \Z^{m-r}$; that is, each simplex of $K$ is mapped to a part of a lattice basis of $\Z^{m-r}$. We denote it by $s(K)$. Replacing $\Z$ by $\Z_p$ for a prime $p\in\N$ and the integer lattice by a $\Z_p$-vector space above, we get a definition of the mod $p$ Buchstaber number of $K$. We denote it by $s_p(K)$. The theory of complex and real Buchstaber numbers has been developed in the framework of toric topology and combinatorial commutative algebra, see~\cite{Ayz10, Ayz16, Er08, Er14, BG}.

These invariants play a crucial role in toric topology~\cite{BP02,TT} as they reflect important topological properties of polyhedral products~\cite{BBCG}. Namely, the complex Buchstaber number of $K$ equals the maximal rank of a toric subgroup $H$ in the compact torus $\T^m\subset (\D^2)^m$ such that the restricted coordinatewise action of $H$ on the moment-angle-complex $\zk\subseteq (\D^2)^m$ is free. Similarly, one can interpret the mod $p$ Buchstaber numbers, see~\cite{BVV}.

It turns out that chromatic and Buchstaber numbers of an $(n-1)$-dimensional simplicial complex $K$ with $m$ vertices are closely related to each other by the following inequalities:
$$
m-\chi(K)\leq s(K)\leq s_p(K)\leq m-n\,\text{ for any prime }p\in\N,
$$
which we discuss and use in Section 4. The case of the maximal Buchstaber invariant $s(K)=m-n$ and a starshaped sphere $K$ is of particular importance in toric topology; then the $(m+n)$-dimensional cellular space $\zk$ acquires an equivariant smooth structure and the above mentioned orbit space construction yields a stably complex $2n$-dimensional closed, orientable manifold $M(K,\Lambda)=\zk/H$, which generalizes the notion of a quasitoric manifold corresponding to a polytopal sphere $K$, see~\cite{TT}. Classification of these manifolds and, in particular, a combinatorial description of all starshaped spheres with a maximal Buchstaber number remains an open problem of equivariant topology.

In Section 4 we show that all Bier spheres belong to this class of starshaped spheres, which gives rise to a new wide class of quasitoric manifolds. Indeed, due to~\cite{Er08}, $\chi(K)\in\{n, n+1\}$ for a polytopal sphere $K$ implies that $s(K)=m-n$; then Theorem~\ref{BierChromaticTheorem} (in the polytopal case) and our construction of a canonical characteristic map (in the general case) yield the following general formula, see Theorem~\ref{BuchNumTheo}, for all the Buchstaber numbers of $\Bier(K)$ in terms of the $f$-vector of $K$ only.

\begin{theorem}
Let $K\neq\Delta_{[m]}$ be a simplicial complex on $[m]$ with $m\geq 2$. Then we have
$$
s(\Bier(K))=s_p(\Bier(K))=f_0(K)-f_{m-2}(K)+1\,\text{ for any prime }p\in\N.
$$
\end{theorem}

In our subsequent publications, we are going to use the previous result in order to show that our canonical characteristic map $\Lambda_K$ for a Bier sphere $\Bier(K)$ determines the rays of a complete regular fan, whose underlying simplicial complex is isomorphic to $\Bier(K)$. Therefore, the corresponding manifold $M(\Bier(K),\Lambda_K)$ acquires a structure of a toric manifold. Taking $\Bier(K)$ to be a non-polytopal Bier sphere, this immediately solves another problem of equivariant topology: to show the existence of infinitely many toric, but not quasitoric manifolds.

Finally, the chordality of Bier spheres is discussed in Section 4. Recall that a simple graph is chordal if it contains no induced subgraphs being cycles of length greater than three. We call a simplicial complex chordal if its 1-skeleton is a chordal graph. It turns out that this combinatorial property acquires the following algebraic and topological interpretations in the framework of toric topology. 

We are interested here in the following open problem of algebraic topology: classify all closed connected manifolds $M$ with $\cat(M)=2$, where $\cat(X)$ denotes the Lusternik-Schnirelmann category of a topological space $X$. It is well-known that $\cat(X)=1$ if and only if $X$ is a co-H-space, hence $\cat(\zk)=1$ for a starshaped sphere $K$ if and only if $K$ is a boundary of a simplex. Thus, our next goal is to describe the class of moment-angle manifolds $\zk$ with $\cat(\zk)=2$ in combinatorial and algebraic terms. By the result of~\cite{DKR}, a closed connected $n$-manifold $M$ with $n>2$ has $\cat(M)=2$ only if $\pi_1(M)$ is a free group. Furthermore, it was shown in~\cite{PV16} that the fundamental group $\pi_1(\rk)$ of a real moment-angle-complex $\rk$ is isomorphic to the commutator subgroup in the right-angled Coxeter group of $\sk^1(K)$ and the latter is a free group if and only if $K$ is chordal. Therefore, given a starshaped sphere $K$ of dimension $n-1>1$, if the real moment-angle manifold $\rk$ of dimension $n$ has $\cat(\rk)=2$, then $K$ is chordal. 

On the other hand, due to~\cite{OP}, for a 3-dimensional simplicial complex $K$, the cohomology ring $H^*(\zk)$ of its moment-angle-complex $\zk$ is isomorphic to the cohomology ring of a connected sum of products of spheres if and only if either (i) $K$ is the nerve complex of the 4-cube $I^4$, or (ii) $K$ is chordal, or (iii) $K$ has only two minimal non-faces and they form a chordless 4-cycle. These results allow us to formulate the next conjecture similar in a sense to the statement of~\cite[Theorem 4.6]{GPTW}, which asserts that a flag complex $K$ yields a Golod face ring $\Bbbk[K]$ over any field $\Bbbk$ if and only if $K$ is chordal if and only if its moment-angle-complex $\zk$ is a homotopy wedge of spheres and hence $\cat(\zk)=1$. 
\begin{conjecture}\label{conj}
Let $K$ be an $(n-1)$-dimensional chordal starshaped sphere with $n>2$. Then the next conditions are equivalent:
\begin{enumerate}
\item $\cat(\zk)=2$;
\item $\cat(\rk)=2$;
\item $\zk$ is homeomorphic to a connected sum of sphere products with two spheres in each product;
\item $\rk$ is homeomorphic to a connected sum of sphere products with two spheres in each product;
\item $\Bbbk[K]$ is minimally non-Golod over any field $\Bbbk$.
\end{enumerate}
\end{conjecture}

In this paper, we are going to make a first step towards the proof of the above conjecture in the class of moment-angle manifolds over Bier spheres. Namely, in~\Cref{ChordalBierTheorem}, we classify all chordal and all stacked Bier spheres by proving the following result.

\begin{theorem}
Let $K\neq\Delta_{[m]}$ be a simplicial complex on $[m]$ with $m\geq 2$. Then the following three statements are equivalent:
\begin{itemize}
\item[(a)] $\Bier(K)$ is a chordal complex;
\item[(b)] $\Bier(K)$ is a polytopal sphere and one of the next three cases holds:
\begin{enumerate}
\item $P_K=\Delta^1$, where $m=2$;
\item $P_K=\Delta^2$, where $m=3$;
\item $P_K\in\vc^{k}(\Delta^{m-1})$, where $0\leq k\leq m$ and $m\geq 4$. \end{enumerate}
\item[(c)] One of the next three cases holds:
\begin{enumerate}
\item $m=2$;
\item $m=3$ and either $K$ or $K^\vee$ equals $\partial\Delta^2$;
\item $m\geq 4$ and either $K$ or $K^\vee$ has no edges.
\end{enumerate}
\end{itemize}
If $m\leq 3$, then $\Bier(K)$ is a stacked sphere being either the boundary of a simplex, or the $k$-cycle with $k=4, 5, $ or $6$. If $m\geq 4$, then $P_K$ is obtained by cutting
off $k$ vertices of the simplex $\Delta^{m-1}$, where $k$ is the number of vertices of the corresponding complex without edges; in this case, {\rm (a)}, {\rm (b)} and {\rm (c)} are equivalent to:
\begin{itemize}
\item[(d)] $\Bier(K)$ is a stacked sphere.
\end{itemize}
\end{theorem}

Here, by $\vc^{k}(\Delta^{m-1})$ we mean a set of (combinatorial) truncation polytopes; each of them can be obtained from $\Delta^{m-1}$ by a sequence of $k$ vertex cuts by hyperplanes in general position. It follows from~\cite[Theorem 8.5]{Kalai} that a simplicial $(d-1)$-sphere with $d\geq 3$ is a stacked sphere (i.e. is isomorphic to the nerve complex of a truncation polytope) if and only if it is chordal and has no minimal non-faces of cardinality $k$ for $2 < k < d$. Obviously, in case $m\leq 3$, all Bier spheres are stacked, see~\cite[Example 2.8]{LZ}. As \Cref{ChordalKeyLemma} asserts, in case $m\geq 4$, the chordality of $\Bier(K)$ is equivalent to either $K$ or $K^\vee$ having no edges. This yields a canonical realization for the corresponding Bier sphere as a boundary of a stacked polytope; that is, one gets $\Bier(K)$ in that case as a boundary of a simplicial $(m-1)$-dimensional polytope obtained from the simplex $\Delta^{m-1}$ by performing stellar subdivisions (gluing pyramids) over some of its facets. Based on this classification, we are going to prove Conjecture~\ref{conj} for all chordal Bier spheres in our subsequent publications.

\section{Basic definitions and constructions}

In this section, we introduce our notations and recall some necessary definitions, examples, and results related to the Bier sphere construction as well as to the notion of chromatic number for simple graphs and, more generally, for simplicial complexes. 

\begin{definition}
Set $[m]:=\{1,2,\ldots,m\}$ for an integer $m\geq 1$. We say that a non-empty subset $K\subseteq 2^{[m]}$ is an {\emph{abstract simplicial complex on}} $[m]$ if the following condition holds:
$$
\sigma\in K, \tau\subseteq\sigma\Rightarrow \tau\in K.
$$
Elements of $K$ are called its {\emph{faces}} or {\emph{simplices}}. Faces of cardinality $1$ are called (geometrical) {\emph{vertices}}, and faces of cardinality $2$ are called {\emph{edges}}.

The set of vertices of $K$ will be denoted by $V(K)$. Maximal (by inclusion) simplices of $K$ are called its {\emph{facets}}, and their set is denoted by $\M(K)$; we say that facets of $K$ \emph{generate} $K$. 

Minimal (by inclusion) subsets in $[m]$ that are not simplices of $K$ are called its {\emph{minimal non-faces}}, and their set is denoted by $\MF(K)$. Minimal non-faces of cardinality $1$ are called {\emph{ghost vertices}}.

Recall that the {\emph{dimension}} of a simplex is one less than its cardinality; the {\emph{dimension}} of a simplicial complex $K$ equals the maximal dimension of its simplex and is denoted by $\dim(K)$. If all maximal faces of $K$ have the same cardinality, then $K$ is called {\emph{pure}}.
\end{definition}

One of the most important {\emph{combinatorial invariants}} (i.e. characteristics invariant under a face lattice isomorphism) of a simplicial complex is its face vector. Namely, the $f$-\emph{vector} of a simplicial complex $K$ on $[m]$ is a tuple $f(K):=(f_{-1}(K),\ldots,f_{m-1}(K))$, where $f_i(K)$ denotes the number of faces of dimension $i$ in $K$. Note that $f_{-1}(K)=1$ for any simplicial complex $K$, since $\varnothing\in K$ is the only simplex in $K$ of dimension $-1$, and $f_i(K)=0$ if $i>\dim(K)$.

\begin{example}
Let $P$ be a convex simple $n$-dimensional polytope with $m$ facets; that is, a bounded intersection of $m$ closed halfspaces in $\R^n$ being in general position. Then the boundary $\partial P^*$ of its dual simplicial polytope $P^*$ is a geometric realization of a certain abstract simplicial complex which we denote by $K_P$ and call the nerve complex of $P$. Thus, we have:
$$
f_0(K_P)=m\text{ and }\dim(K_P)=n-1.
$$
\end{example}

\begin{example}\label{OneDimBierExample}
Let $m\ge 3$ and $P_m\subset\RR^2$ be an $m$-gon. Then its dual polytope $P_m^*$ is isomorphic to $P_m$. Therefore, $P_m$ is a simple and a simplicial polytope at the same time. The nerve complex of $P_m$ equals its boundary $\partial P_m$, which is the $m$-cycle denoted in what follows by $Z_m$.
\end{example}

\begin{figure}[ht]
\begin{tikzpicture}
\path 
  coordinate (A) at (210:1)
  coordinate (B) at (-30:1)
  coordinate (C) at (90:1);
\draw (A) -- (B) -- (C) -- cycle;
\fill[black] (A) circle (0.05)
             (B) circle (0.05)
             (C) circle (0.05);
\draw (-90:0.9) node{$Z_3$};              
\end{tikzpicture}
\hskip10mm
\begin{tikzpicture}
\path 
  coordinate (A) at (225:1)
  coordinate (B) at (-45:1)
  coordinate (C) at (45:1)
  coordinate (D) at (135:1);
\draw (A) -- (B) -- (C)--(D)--cycle;
\fill[black] (A) circle (0.05)
             (B) circle (0.05)
             (C) circle (0.05)
             (D) circle (0.05);
\draw (-90:1.1) node{$Z_4$}; 
\end{tikzpicture}
\hskip10mm
\begin{tikzpicture}
\path 
  coordinate (A) at (90:1)
  coordinate (B) at (162:1)
  coordinate (C) at (234:1)
  coordinate (D) at (306:1)
  coordinate (E) at (378:1);
\draw (A)--(B)--(C)--(D)--(E)--cycle;
\fill[black] (A) circle (0.05)
             (B) circle (0.05)
             (C) circle (0.05)
             (D) circle (0.05)
             (E) circle (0.05);
\draw (-90:1.2) node{$Z_5$}; 
\end{tikzpicture}
\hskip10mm
\begin{tikzpicture}
\path 
  coordinate (A) at (0:1)
  coordinate (B) at (60:1)
  coordinate (C) at (120:1)
  coordinate (D) at (180:1)
  coordinate (E) at (240:1)
  coordinate (F) at (300:1);
\draw (A)--(B)--(C)--(D)--(E)--(F)--cycle;
\fill[black] (A) circle (0.05)
             (B) circle (0.05)
             (C) circle (0.05)
             (D) circle (0.05)
             (E) circle (0.05)
             (F) circle (0.05);
\draw (-90:1.2) node{$Z_6$}; 
\end{tikzpicture}
\caption{The only cycles that are 1-dimensional Bier spheres.}\label{fig:1BierSpehers}
\end{figure}
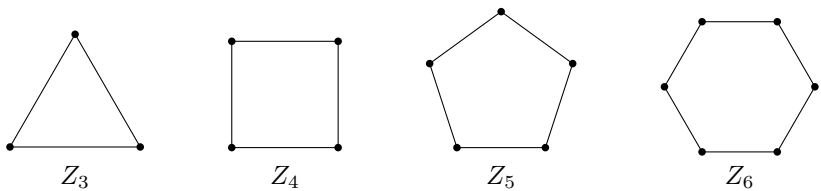

\begin{example}
Consider the following simplicial complexes on $[m]$ with $m\geq 1$:
\begin{itemize}
\item $\varnothing_{[m]}:=\{\varnothing\}$, the void complex; 
\item $\Delta_{[m]}:=2^{[m]}$, the simplex;
\item $\partial \Delta_{[m]}:=2^{[m]}\setminus\{[m]\}$, the boundary of the simplex.
\end{itemize}
Note that $\dim(\varnothing_{[m]})=-1$, $\dim(\Delta_{[m]})=m-1$, and $\dim(\partial\Delta_{[m]})=m-2$, for any $m\geq 1$. Moreover, all the three simplicial complexes are pure and $\MF(\Delta_{[m]})=\varnothing$, $\MF(\partial\Delta_{[m]})=\{[m]\}$, while $\varnothing_{[m]}$ has $m$ ghost vertices. 
\end{example}

In what follows, we sometimes interchangeably use the notations $i$ and $\{i\}$ for a (ghost) vertex of a simplicial complex $K$ on $[m]$ when $i\in [m]$.

\begin{definition}
Let $K$ be a simplicial complex on $[m]$ and $I\subseteq [m]$. We call the simplicial complex
$$
K_I:=\{J\in K\,|\,J\subseteq I\}=K\cap 2^I
$$
the \emph{full subcomplex} in $K$ on the set $I$. In particular, if $i\in [m]$, then the \emph{deletion} of $i$ from $K$ is the full subcomplex $\del_{K}(i)$ in $K$ on the set $[m]\setminus\{i\}$. If $I\in K$ we call the simplicial complex
$$
\link_K(I):=\{J\in K\,|\,I\cup J\in K, I\cap J=\varnothing\}
$$
the {\emph{link}} in $K$ of the face $I$.
\end{definition}

\begin{example}
Let $K$ be a simplicial complex on $[m]$. Observe that the following equivalences hold:
\begin{itemize}
\item $I\in K$ if and only if $K_I=\Delta_{I}$;
\item $I\in \MF(K)$ if and only if $K_I=\partial\Delta_{I}$;
\item $\link_K(I)=K$ if and only if $I=\varnothing$;
\item $\link_K(I)=\{\varnothing\}$ if and only if $I\in M(K)$.
\end{itemize}
\end{example}

A particularly important class of simplicial complexes is formed by all (simple) graphs. Throughout this paper, by a \emph{graph} $\Gamma=\Gamma(V,E)$ with the \emph{vertex set} $V$ and the \emph{edge set} $E\subseteq \binom{V}{2}$ we mean a simplicial complex on $V$ of dimension not greater than $1$ and without ghost vertices. 

The \emph{$n$-skeleton} $\sk^n(K), n\geq 0$ of a simplicial complex $K$ is a simplicial complex with no ghost vertices consisting of all simplices of $K$ of dimension not greater than $n$. Note that both $\sk^0(K)$ and $\sk^1(K)$ are graphs. Moreover, one has: 
$$
V(\sk^0(K))=V(K)=V(\sk^1(K))\text{ and }\sk^0(K)=V(K)\sqcup\{\varnothing\},
$$
for any simplicial complex $K$.

Here is a definition of a simplicial operation, which will be used frequently in what follows.

\begin{definition}\label{ConeSuspensionDef}
Let $K_1$ be a complex on $V_1$ and $K_2$ be a complex on $V_2$ with $V_1\cap V_2=\varnothing$. Then we define their \emph{join} $K$ as a simplicial complex on $V_1\sqcup V_2$ such that
$$
K=K_1\ast K_2:=\{\sigma_1\sqcup\sigma_2\,|\,\sigma_i\in K_i,\, i=1,2\}.
$$
In particular, we have 
\begin{itemize}
\item $\cone_v(K):=\{v\}\ast K$, the \emph{cone} over $K$ with \emph{apex} $v$; 
\item $\Sigma_{\{v,w\}}(K):=\partial\Delta_{\{v,w\}}\ast K$, the \emph{suspension} over $K$ with \emph{vertex pair} $\{v,w\}$.
\end{itemize}    
\end{definition}

In particular, if $K=\cone_{v}(L)$ is a cone over $L$ with apex $v$, then $L=\del_{K}(v)$ is a deletion of $v$ from $K$.

\begin{definition}
Let $\Gamma=\Gamma(V,E)$ be a graph with the vertex set $V$ and the edge set $E$. A \emph{proper coloring} of $\Gamma$ is a surjective map $c\colon V\to C$ such that
$$
\text{if}\quad\{i,j\}\in E,\quad\text{then}\quad c(i)\ne c(j).
$$
The \emph{chromatic number} $\chi(\Gamma)$ of the graph $\Gamma$ is the smallest cardinality of the \emph{set of colors} $C$ such that there exists a proper coloring $c\colon V\to C$.
We set the {\emph{chromatic number}} $\chi(K)$ of a simplicial complex $K$ to be equal to $\chi(\sk^1(K))$ and we say that $c$ is a \emph{proper coloring} of $K$ if it is a proper coloring of $\sk^1(K)$. Finally, we set the chromatic number $\chi(P)$ of a simple polytope $P$ to be equal to $\chi(K_P)$ and, similarly, we introduce the notion of a {\emph{proper coloring}} of $P$.    
\end{definition}

The next result is a direct corollary from the definition of a chromatic number and is well-known.

\begin{proposition}\label{BasicChromaticNumbersProp}
The following statements hold.
\begin{itemize}
\item[(a)] Let $m\geq 3$. Then one has:
$$
\chi(P_m)=2,\text{ if $m$ is even and }\chi(P_m)=3,\text{ if $m$ is odd.}
$$
\item[(b)] For any simplicial complex $K$ and a disjoint point $v$ one has: 
$$
\chi(\cone_v(K))=\chi(K)+1.
$$
\item[(c)] For any simplicial complex $K$ and two distinct points, $v$ and $w$, disjoint from $K$ one has: 
$$
\chi(\Sigma_{\{v,w\}}(K))=\chi(K)+1.
$$
\end{itemize}
\end{proposition}

Now we are going to introduce the main object of study in this paper, a Bier sphere. In the definition below, we use the following notations. Let $[m]:=\{1,2,\ldots,m\}$ and $[m']:=\{1',2',\ldots,m'\}$ be two disjoint ordered sets with the map $\phi\colon i\mapsto i', 1\leq i\leq m$ being an order preserving bijection between them. Denote by $I'$ the image $\phi(I)$ of a subset $I\subseteq [m]$.
 
\begin{definition}
Suppose $K\neq\Delta_{[m]}$ is a simplicial complex on $[m]$ with $m\geq 2$. Then 
\begin{itemize}
\item the {\emph{Alexander dual}} of $K$ is a simplicial complex $K^\vee$ on $[m']$ such that 
$$
I\in \MF(K)\Longleftrightarrow [m']\setminus I'\in \M(K^\vee);
$$
\item the {\emph{Bier sphere}} of $K$ is a simplicial complex $\Bier(K)$ on $[m]\sqcup [m']$ such that 
$$
\Bier(K)=\{I\sqcup J'\,|\,I\in K, J'\in K^\vee, I\cap J=\varnothing\};
$$
that is, $\Bier(K)$ is the {\emph{deleted join}} of $K$ and $K^\vee$.
\end{itemize}
\end{definition}

It was shown in~\cite{Bier} that $\Bier(K)$ defined above is a PL-sphere of dimension $(m-2)$ with the number of vertices varying between $m$ and $2m$. It follows from the definition that the Bier sphere construction is symmetric: $\Bier(K)=\Bier(K^\vee)$. In \Cref{fig:1BierSpehers} one can find the geometric realizations of all possible 1-dimensional Bier spheres, see also~\cite[Example 2.8]{LZ}.

\begin{example}\label{ex:simplex}
Let $K_1=\Delta_{[m-1]}$ and $K_2=\Delta_{[m-2]}$ be simplicial complexes on $[m]$, where $m\ge 3$. Then $K_1^\vee=\Delta_{[(m-1)']}$ and $K_2^\vee= \Delta_{[(m-1)']}\cup \Delta_{[(m-2)']\cup\{m'\}}$. Observe that a maximal simplex in $\Bier(K_1)$ is of a form $([m-1]\setminus I)\sqcup I'$, where $I\subseteq [m-1]$, hence
$$
\Bier(K_1)=\partial\Delta_{\{1,1'\}}\ast\partial\Delta_{\{2,2'\}}\ast\cdots\ast\partial\Delta_{\{m-2,(m-2)'\}}\ast\partial\Delta_{\{m-1,(m-1)'\}}.
$$
Furthermore, a maximal simplex in $\Bier(K_2)$ is of a form $([m-2]\setminus I)\sqcup I'\sqcup\{k'\}$, where $I\subseteq [m-2]$ and $k\in\{m-1,m\}$, hence
$$
\Bier(K_2)=\partial\Delta_{\{1,1'\}}\ast\partial\Delta_{\{2,2'\}}\ast\cdots\ast\partial\Delta_{\{m-2,(m-2)'\}}\ast\partial\Delta_{\{(m-1)',m'\}}.
$$
Thus, we see that both $\Bier(K_1)$ and $\Bier(K_2)$ are $(m-3)$-times iterated suspensions over the 4-cycle $Z_4$, or, in other words, both are isomorphic to the boundary of the $(m-1)$-dimensional cross-polytope $(I^{m-1})^*$.
\end{example}

Taking into account \Cref{BasicChromaticNumbersProp} and \Cref{ex:simplex}, our first goal is to study the chromatic numbers of Bier spheres. When $m=2$, the only Bier sphere has a geometric realization as the set of two disjoint points, it has chromatic number $1$. Therefore, in what follows, if otherwise not stated explicitly, we always assume that a simplicial complex on $[m]$ under consideration is different from $\Delta_{[m]}$ and that $m$ is always greater than or equal to $3$.

When $m=3$, it is clear that the Bier sphere $\Bier(K)$ is a boundary of a polygon, and so its chromatic number is equal to either $2$ or $3$, depending on the parity of the number of vertices, see \Cref{BasicChromaticNumbersProp} (a). Up to isomorphism and taking Alexander dual, we have only five simplicial complexes on [3] shown in \Cref{fig:1complexes}.
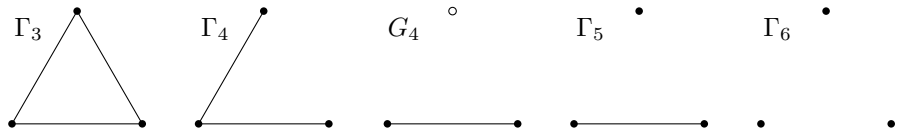
\begin{figure}[ht]
\begin{tikzpicture}
\path 
  coordinate (A) at (210:1)
  coordinate (B) at (-30:1)
  coordinate (C) at (90:1);
\draw (A) -- (B) -- (C) -- cycle;
\fill[black] (A) circle (0.05)
             (B) circle (0.05)
             (C) circle (0.05);
\draw (130:1) node{$\Gamma_3$};              
\end{tikzpicture}
\hskip5mm
\begin{tikzpicture}
\draw (B) -- (A) -- (C);
\fill[black] (A) circle (0.05)
             (B) circle (0.05)
             (C) circle (0.05);
\draw (130:1) node{$\Gamma_4$}; 
\end{tikzpicture}
\hskip5mm
\begin{tikzpicture}
\draw (B) -- (A);
\fill[black] (A) circle (0.05)
             (B) circle (0.05);
\draw[black] (C) circle (0.05);
\draw (130:1) node{$G_4$}; 
\end{tikzpicture}
\hskip5mm
\begin{tikzpicture}
\draw (B) -- (A);
\fill[black] (A) circle (0.05)
             (B) circle (0.05)
             (C) circle (0.05);
\draw (130:1) node{$\Gamma_5$}; 
\end{tikzpicture}
\hskip5mm
\begin{tikzpicture}
\fill[black] (A) circle (0.05)
             (B) circle (0.05)
             (C) circle (0.05);
\draw (130:1) node{$\Gamma_6$}; 
\end{tikzpicture}
\caption{The complexes $G_4$ and $\Gamma_6$ are self-dual, $G_4$ has a ghost vertex. The indices in the notation of a complex indicate the number of vertices in the corresponding Bier sphere.}\label{fig:1complexes}
\end{figure}
The chromatic number of Bier spheres $\Bier(\Gamma_4)$, $\Bier(G_4)$, and $\Bier(\Gamma_6)$ is 2, since it is $Z_4$ in the first two cases and $Z_6$ in the last case. The chromatic number of Bier spheres $\Bier(\Gamma_3)$ and $\Bier(\Gamma_5)$ is 3, since they are isomorphic to $Z_3$ and $Z_5$, respectively. 
This motivates us to compute the chromatic number of an arbitrary Bier sphere: it is done in the next section.

\section{Bier spheres and their chromatic numbers}

In this section, we show that every Bier sphere $\Bier(K)$ of a simplicial complex $K$ on $[m]$ with $m\geq 2$ has the chromatic number equal to either $m-1$, or $m$ and describe all the simplicial complexes $K$, for which $\chi(\Bier(K))=m-1$. The discussion of chromatic numbers of 1-dimensional Bier spheres in the end of the previous section as well as \Cref{ex:simplex} will serve as the first step in the desired description.

\begin{proposition}\label{ChromaticNumberBasicEstimateProp}
Let $K$ be a simplicial complex on $[m]$. Then
$$
\max(m-1, \chi(K), \chi(K^\vee))\leq \chi(\Bier(K))\leq m.
$$
\end{proposition}
\begin{proof}
Since $\dim(\Bier(K))=m-2$, the chromatic number $\chi(\Bier(K))\geq m-1$. On the other hand, the map
$$
i, i'\mapsto c_i\text{ for }1\leq i\leq m,
$$
yields a proper coloring of the vertex set of $\Bier(K)$ by the set of colors $C=\{c_1,\ldots,c_m\}$, hence $\chi(\Bier(K))\leq m$.

Finally, since both $K$ and $K^\vee$ are full subcomplexes in $\Bier(K)$ by definition of a Bier sphere, one has
$$
\chi(\Bier(K))\geq \chi(K), \chi(K^\vee),
$$
which finishes the proof.
\end{proof}

Next, we are going to introduce two classes of simplicial complexes that contain the classes of cones and suspensions, see \Cref{ConeSuspensionDef}, respectively.

\begin{definition}
We say that a graph $\Gamma$ with the vertex set $V$ is 
\begin{itemize}
\item a \emph{cone graph} if there exists $a\in V$ such that: 
$$
\link_{\Gamma}(a)=V\setminus \{a\}.
$$
We call a simplicial complex $K$ a \emph{weak cone} if its 1-skeleton $\sk^1(K)$ is a cone graph;
\item a \emph{suspension graph} if there exist distinct $a, b\in V$ such that: 
$$
\link_{\Gamma}(a)=V\setminus \{a,b\}=\link_{\Gamma}(b). 
$$
We call a simplicial complex $K$ a \emph{weak suspension} if its 1-skeleton $\sk^1(K)$ is a suspension graph.   
\end{itemize}
\end{definition}

\begin{example}
Any geometric realization of a simplicial sphere is a topological suspension. On the other hand, the boundary of an octahedron is a suspension and, in particular, a weak suspension with respect to each pair of its opposite vertices, whereas both the boundaries of a tetrahedron and an icosahedron are not even weak suspensions, see \Cref{fig:platonic}.    
\end{example}

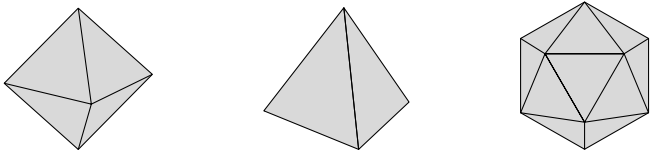
\begin{figure}[ht]
\tdplotsetmaincoords{70}{80}
\begin{tikzpicture}[tdplot_main_coords,line join=round]
\draw[dashed] (-1,0,0)--(0, 0, 1)
              (-1,0,0)--(0, 0, -1)
              (-1,0,0)--(0, 1, 0)
              (-1,0,0)--(0, -1, 0);
\fill[gray!30,opacity=0.5] (0, 1, 0)--(0, 0, 1)--(0,-1,0)--(0, 0, -1)--cycle;
\draw (1,0,0)--(0, 0, 1)
      (1,0,0)--(0, 0, -1)
      (1,0,0)--(0, 1, 0)
      (1,0,0)--(0, -1, 0);
\draw (0, 1, 0)--(0, 0, 1)--(0,-1,0)--(0, 0, -1)--cycle;      
\end{tikzpicture}
\hskip1cm
\tdplotsetmaincoords{70}{80}
\begin{tikzpicture}[tdplot_main_coords,line join=round,scale=1.2]
\pgfmathsetmacro\a{sqrt(2/9)}
\pgfmathsetmacro\b{sqrt(2/3)}
\pgfmathsetmacro\c{1/3}
\draw[dashed] (-\a, \b, -\c)--(-\a, -\b, -\c);
\fill[gray!30,opacity=0.5] (2*\a, 0, -\c)--(-\a, \b, -\c)--(0,0,1)--(-\a, -\b, -\c)--cycle;
\draw (2*\a, 0, -\c)--(0,0,1)--(-\a, -\b, -\c)--cycle;
\draw (2*\a, 0, -\c)--(-\a, \b, -\c)--(0,0,1)--cycle;
\end{tikzpicture}
\hskip1cm
\tdplotsetmaincoords{70}{90}
\begin{tikzpicture}[tdplot_main_coords,line join=round]
    \pgfmathsetmacro\a{sqrt(5-sqrt(5))/sqrt(10)}
    \pgfmathsetmacro{\phi}{\a*(1+sqrt(5))/2}
    \path 
    coordinate(A) at (0,\phi,\a)
    coordinate(B) at (0,\phi,-\a)
    coordinate(C) at (0,-\phi,\a)
    coordinate(D) at (0,-\phi,-\a)
    coordinate(E) at (\a,0,\phi)
    coordinate(F) at (\a,0,-\phi)
    coordinate(G) at (-\a,0,\phi)
    coordinate(H) at (-\a,0,-\phi)
    coordinate(I) at (\phi,\a,0)
    coordinate(J) at (\phi,-\a,0)
    coordinate(K) at (-\phi,\a,0)
    coordinate(L) at (-\phi,-\a,0); 
    \draw[dashed] (A) -- (E) -- (C)
    (B) -- (I) -- (J) -- (D)
    (C) -- (J) -- (E) -- (I) -- (A)
    (J) -- (F) -- (I)
    (G) -- (E);
    \fill[gray!30,opacity=0.5] (A) -- (B) -- (F) -- (D) -- (C) -- (G) -- cycle;
    \draw (B) -- (H) -- (F) 
    (D) -- (L) -- (H) --cycle 
    (K) -- (L) -- (H) --cycle
    (K) -- (L) -- (G) --cycle
    (C) -- (L) (B)--(K) (A)--(K);
    \draw (A) -- (B) -- (F) -- (D) -- (C) -- (G) -- cycle; 
\end{tikzpicture} 
\caption{The boundary of an octahedron is a suspension and therefore a weak suspension, but the boundaries of a tetrahedron and an icosahedron are not even weak suspensions.}\label{fig:platonic}
\end{figure}

Obviously, any cone over a simplicial complex is a weak cone complex, and any suspension over a simplicial complex is a weak suspension complex. The opposite implications fail; in the next example, we introduce a family of simplicial spheres, one sphere for each $m\geq 5$, such that each member of the family is a weak suspension, but not a suspension.

\begin{example}
Let $m \geq 5$ and consider the following simplicial complex $K_m$ on the vertex set $[m]$:
$$
\MF(K_m)=\{\{1,2,\ldots,m-2\},\{2,3,\ldots,m-1\},\{2,m\},\ldots,\{m-1,m\}\}.
$$
By definition, its Alexander dual complex $K_m^\vee$ is a simplicial complex on the vertex set $[m']$ given by
$$
\M(K_m^\vee)=\{\{1',m'\},\{(m-1)',m'\},[(m-1)']\setminus\{i'\}\,|\,2\leq i\leq m-1\}.
$$
Thus, $K_m^\vee$ is a union of two edges and the cone with apex $1'$ over the standard triangulation of $\partial\Delta_{\{2',\ldots,(m-1)'\}}$.
This implies that the Bier sphere $\Bier(K_m)$ is a weak suspension with respect to the unique vertex pair $\{1,1'\}$. On the other hand, neither $K_m$ nor $K_m^\vee$ is a cone, so $\Bier(K_m)$ is not a suspension.
\end{example}

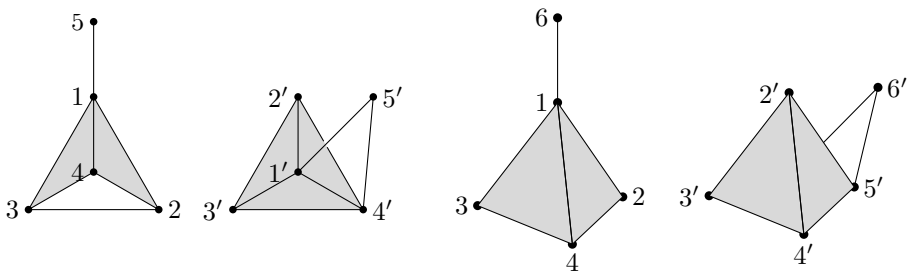
\begin{figure}[ht]
\begin{tikzpicture}
\clip (-1.3,-1.43) rectangle (1.15,2.1);
\fill[gray!30,opacity=0.5] (0,1)--(210:1)--(0,0)--(330:1)--cycle;
\fill (0,1) node[left]{1} circle (0.05)
      (330:1)node[right]{2} circle (0.05)
      (210:1)node[left]{3} circle (0.05)
      (0,0) node[left]{4} circle (0.05)
      (0,2) node[left]{5} circle (0.05);
\draw (0,0)--(0,2);
\draw (210:1)--(330:1)
      (0,1)--(210:1)--(0,0)--(330:1)--cycle;
\end{tikzpicture}
\hskip2mm
\begin{tikzpicture}
\clip (-1.25,-1.43) rectangle (1.4,2.1);
\fill[gray!30,opacity=0.5] (0,1)--(210:1)--(330:1)--cycle;
\fill (0,1) node[left]{$2'$} circle (0.05)
      (330:1)node[right]{$4'$} circle (0.05)
      (210:1)node[left]{$3'$} circle (0.05)
      (0,0) node[left]{$1'$} circle (0.05)
      (1,1) node[right]{$5'$} circle (0.05);
\draw (0,0)--(0,1)
      (210:1)--(330:1)
      (0,1)--(210:1)--(0,0)--(330:1)--(1,1)--(0,0);
\path (330:1)--(0,1) node [inner sep=0pt,pos=.57] (A) {};      
\path (330:1)--(0,1) node [inner sep=0pt,pos=.59] (B) {};
\draw (330:1) --(A) (B)--(0,1);
\end{tikzpicture}
\hskip1mm
\tdplotsetmaincoords{70}{80}
\begin{tikzpicture}[tdplot_main_coords,line join=round,scale=1.2]
\pgfmathsetmacro\a{sqrt(2/9)}
\pgfmathsetmacro\b{sqrt(2/3)}
\pgfmathsetmacro\c{1/3}
\node(T1) at (0,0,1) {};
\node(T2) at (-\a, \b, -\c) {};
\node(T3) at (-\a, -\b, -\c) {};
\node(T4) at (2*\a, 0, -\c) {};
\node(T5) at (0,0,0) {};
\node(T6) at (0,0,2) {};
\fill[tdplot_screen_coords] 
  (T1)node[left]{1} circle [radius=0.05]
  (T2)node[right]{2} circle [radius=0.05]
  (T3)node[left]{3} circle [radius=0.05]
  (T4)node[below]{4} circle [radius=0.05]
  (T5)node[left]{5} circle [radius=0.05]
  (T6)node[left]{6} circle [radius=0.05];
\draw[dashed] (-\a, \b, -\c)--(-\a, -\b, -\c)
              (0,0,0)--(-\a, \b, -\c)
              (0,0,0)--(-\a, -\b, -\c)
              (0,0,0)--(2*\a, 0, -\c)
              (0,0,0)--(0,0,1);
\fill[gray!30,opacity=0.5] (2*\a, 0, -\c)--(-\a, \b, -\c)--(0,0,1)--(-\a, -\b, -\c)--cycle;
\draw (2*\a, 0, -\c)--(0,0,1)--(-\a, -\b, -\c)--cycle;
\draw (2*\a, 0, -\c)--(-\a, \b, -\c)--(0,0,1)--cycle;
\draw (0,0,1)--(0,0,2);
\end{tikzpicture}
\hskip1mm
\begin{tikzpicture}[tdplot_main_coords,line join=round,scale=1.2]
\pgfmathsetmacro\a{sqrt(2/9)}
\pgfmathsetmacro\b{sqrt(2/3)}
\pgfmathsetmacro\c{1/3}
\node(T0) at (0,0,-1.1) {};
\node(T2) at (0,0,1) {};
\node(T5) at (-\a, \b, -\c) {};
\node(T3) at (-\a, -\b, -\c) {};
\node(T4) at (2*\a, 0, -\c) {};
\node(T1) at (0,0,0) {};
\node(T6) at (0,1,1) {};
\fill[tdplot_screen_coords] 
  (T1)node[left]{$1'$} circle [radius=0.05]
  (T2)node[left]{$2'$} circle [radius=0.05]
  (T3)node[left]{$3'$} circle [radius=0.05]
  (T4)node[below]{$4'$} circle [radius=0.05]
  (T5)node[right]{$5'$} circle [radius=0.05]
  (T6)node[right]{$6'$} circle [radius=0.05];
\draw[dashed] (-\a, \b, -\c)--(-\a, -\b, -\c)
              (0,0,0)--(-\a, \b, -\c)
              (0,0,0)--(-\a, -\b, -\c)
              (0,0,0)--(2*\a, 0, -\c)
              (0,0,0)--(0,0,1);
\draw (T1.center)--(T6.center)--(T5.center);              
\fill[gray!30,opacity=0.5] (2*\a, 0, -\c)--(-\a, \b, -\c)--(0,0,1)--(-\a, -\b, -\c)--cycle;
\draw (T4.center)--(T2.center)--(T3.center)--cycle;
\draw (T4.center)--(T5.center)--(T2.center)--cycle;
\end{tikzpicture}
\caption{On the left it is the simplicial complex $K_5$ and its Alexander dual $K_5^\vee$. On the right it is the simplicial complex $K_6$ (there is no face $\{2,3,4,5\}$) and its Alexander dual $K_6^\vee$.}
\end{figure}

\begin{proposition}
Let $m\geq 5$. Then both $K_m$ and $K_m^\vee$ are weak cones, and one has: 
$$
\chi(K_m)=\chi(K_m^\vee)=m-1.
$$
On the other hand, $\chi(\Bier(K_m))=m$.
\end{proposition}
\begin{proof}
Consider the following coloring $c\colon[m]\sqcup[m']\to C(m)=\{c_1,\ldots,c_{m-1}\}$ defined by
$$
c(i)=c(i')= c_i\text{ for }1\leq i\leq m-1\text{ and }c(m)=c(m')= c_2.
$$
It is easy to see that $C(m)$ provides a proper coloring of both $K_m$ and $K_m^\vee$ with $(m-1)$ colors, and so
$$
\chi(K_m), \chi(K_m^\vee)\leq m-1.
$$ 

On the other hand, the 1-skeleta of the full subcomplexes $K_{m}\cap2^{[m-1]}$ and $K^\vee_{m}\cap2^{[(m-1)']}$ are complete graphs with $(m-1)$ vertices, and therefore 
$$
\chi(K_m), \chi(K_m^\vee)\geq m-1.
$$
Combining the last two inequalities above yields the first statement. The second statement follows from \Cref{ChromaticNumberBasicEstimateProp} and the observation that the 1-skeleta of the full subcomplexes 
$$
\Bier(K_{m})\cap 2^{[m-1]\cup\{m'\}}\text{ and } \Bier(K_{m})\cap 2^{[(m-1)']\cup\{m\}}
$$ 
are complete graphs with $m$ vertices.
\end{proof}

Skeleta of a simplex never yield a Bier sphere, which is a suspension, as shown in the following example. In the case of the 0-skeleton of the $m$-simplex and its Alexander dual, which is the $(m-3)$-skeleton of the $m$-simplex, the Bier sphere is not even a weak suspension, since it does not contain both edges $\{i,j\}$ and $\{i',j'\}$ for each pair of distinct vertices $i,j\in [m]$. The next example shows that in all other cases, the Bier sphere of a skeleton of a simplex is a weak suspension and not a suspension.

\begin{example}
Let $m\ge 5$, $1\le k\le m-4$, and $K=\Delta_{[m]}^{(k)}$, i.e., the set of maximal simplices is $\M(K)=\{I\subset[m]\,|\, |I|=k+1\}$. 
Then $I'$ is a minimal non-face in $K^\vee$ if and only if $|I|=m-k-1$, hence $K^\vee=\Delta_{[m']}^{(m-k-3)}$.
For every $i,j\in[m]$, $i\ne j$, we have $\{i,j'\}$, $\{i',j\}\in\Bier(K)$, since there are no ghost vertices in $K$ and $K^\vee$, $\{i,j\}\in K\subset \Bier(K)$, since $k\ge 1$, and $\{i',j'\}\in K^\vee\subset\Bier(K)$, since $m-k-3\ge 1$. Therefore $\Bier(K)$ is a weak suspension for every pair $\{i,i'\}$.
For every $i\in[m]$ we can write $[m]\setminus\{i\}=I\sqcup J$, where $|I|=k+1$ and $|J|=m-k-2$, hence $I\sqcup J'$ is a maximal simplex in $\Bier(K)$, so $\Bier(K)$ is not a suspension with a vertex pair $\{i,i'\}$ and therefore it is not a suspension. Finally, all three chromatic numbers $\chi(K)$, $\chi(K^\vee)$, and $\chi(\Bier(K))$ are equal to $m$, since the complete graph on $m$ vertices is a subcomplex of all three complexes.
\end{example}

The next result is formally a generalization of \Cref{BasicChromaticNumbersProp} (b) and (c). However, in fact, the two statements are equivalent since the chromatic number of a simplicial complex depends only on its 1-skeleton, so we omit the proof.

\begin{lemma}\label{ChromNumOfWeakSuspensionLemma}
Suppose a simplicial complex $L$ is a weak suspension or a weak cone over a simplicial complex $K$. Then we have:
$$
\chi(L)=\chi(K)+1.
$$
\end{lemma}

In the following lemma, we analyze vertex pairs of a Bier sphere being a weak suspension. 

\begin{lemma}\label{SuspensionStructureLemma}
Let $K\neq\Delta_{[m]}$ be a simplicial complex on $[m]$ with $m\geq 3$. Suppose $\Bier(K)$ is a weak suspension. Then there exists an $i\in [m]$ such that $\{i,i'\}$ is a vertex pair of $\Bier(K)$.
\end{lemma}
\begin{proof}
Denote by $\Gamma=\sk^1(\Bier(K))$. Assume $\{a,b\}\subseteq V(K)$ is a vertex pair of $\Bier(K)$. Then $\{a,b\}\in\MF(K)$ and $\{a'\},\{b'\}\in\MF(K^\vee)$. Hence 
$$
[m]\setminus\{a\}, [m]\setminus\{b\}\in K,\text{ and }
[m']\setminus\{a',b'\}\in K^\vee.
$$

The latter implies that $K^\vee$ is isomorphic to $\Delta_{(m-2)'}$, since $\{a'\},\{b'\}\in\MF(K^\vee)$. Then Example~\ref{ex:simplex} shows the statement is true in this case. The case where $\{a',b'\}\subseteq V(K^\vee)$ is a vertex pair of $\Bier(K)$ is similar. 
\end{proof}

\begin{proposition}\label{BierWeakSuspCriterionProp}
Let $K$ be a simplicial complex on $[m]$ and $i\in [m]$. Then $\Bier(K)$ is a weak suspension with a vertex pair $\{i,i'\}$ if and only if $K$ is a weak cone with apex $i$ and $K^\vee$ is a weak cone with apex $i'$.
\end{proposition}
\begin{proof}
If $\Bier(K)$ is a weak suspension with a vertex pair $\{i,i'\}$, then $K$ and $K^\vee$ are weak cones with apex $i$ and $i'$, respectively, since they are subcomplexes.
On the other hand, if $K$ and $K^\vee$ are weak cones with apex $i$ and $i'$, respectively, then for any geometrical vertices $j$ and $k'$ of $\Bier(K)$ with $j,k\ne i$, the edges $\{i,j\}$ and $\{i',k'\}$ are in $\Bier(K)$ and, by the Bier sphere construction, also the edges $\{i,k'\}$ and $\{i',j\}$.
\end{proof}

\begin{lemma}\label{BierSphereAsSuspensionLemma}
Let $K$ be a simplicial complex on $[m]$ and $i\in [m]$. Then the following statements are equivalent.
\begin{itemize}
    \item[(a)] The complex $K$ is a cone with an apex $i$.
    \item[(b)] The complex $K^\vee$ is a cone with an apex $i'$.
    \item[(c)] The Bier sphere $\Bier(K)$ is a suspension with a vertex pair $\{i,i'\}$.
\end{itemize}
\end{lemma}
\begin{proof}
The equivalence (a) $\Leftrightarrow$ (b) follows from the observation that a vertex $i$ in $K$ belongs to each element in $\M(K)$ if and only if $i'$ does not belong to any element in $\MF(K^\vee)$ if and only if $i'$ belongs to each element in $\M(K^\vee)$.

The implication (a) $\Rightarrow$ (c) follows from~\cite[Lemma 3.3 (c)]{LZ}, where it was shown that 
$$
\Bier(\cone_{i}(K))=\Sigma_{\{i,i'\}}(\Bier(K)).
$$

The implication (c) $\Rightarrow$ (a) holds, since each maximal simplex in $\Bier(K)$ belonging to $2^{[m]}$ contains the vertex $i$, hence $K=\Bier(K)\cap 2^{[m]}$ is a cone with apex $i$, see \Cref{fig:Biersuspension}.
\end{proof}

The example in the following figure serves as a motivation for the next lemma.
\begin{figure}[ht]
\begin{tikzpicture}
\clip (-1.5,-1.2) rectangle (1.5,1.4);
\fill (-30:1) node[right]{$2$} circle (0.05) 
      (90:1) node[above]{$3$} circle (0.05) 
      (210:1) node[left]{$1$} circle (0.05); 
\end{tikzpicture}
\hskip5mm
\begin{tikzpicture}
\clip (-1.5,-1.6) rectangle (1.5,1.4);
\draw (0:1)--(60:1)--(120:1)--(180:1)--(240:1)--(300:1)--cycle; 
\fill (0:1) node[right]{$1$} circle (0.05) 
      (60:1) node[right]{$2'$} circle (0.05) 
      (120:1) node[left]{$3$} circle (0.05)
      (180:1) node[left]{$1'$} circle (0.05) 
      (240:1) node[left]{$2$} circle (0.05) 
      (300:1) node[right]{$3'$} circle (0.05); 
\end{tikzpicture}
\hskip2cm
\begin{tikzpicture}
\clip (-1.5,-1.2) rectangle (1.5,1.4);
\draw (0,0)--(-30:1) (0,0)--(90:1) (0,0)--(210:1);
\fill (-30:1) node[right]{$2$} circle (0.05) 
      (90:1) node[above]{$3$} circle (0.05) 
      (210:1) node[left]{$1$} circle (0.05)
      (0,0) node[below]{$4$} circle (0.05); 
\end{tikzpicture}
\hskip5mm
\tdplotsetmaincoords{70}{80}
\begin{tikzpicture}[tdplot_main_coords,line join=round]
\pgfmathsetmacro\a{sqrt(3)/2}
\draw[dashed] (-0.5,\a,0)--(-1,0,0)node[below]{$1'$}--(-0.5,-\a,0)
              (-1,0,0)--(0,0,1)node[above]{$4$}
              (-1,0,0)--(0,0,-1)node[below]{$4'$}
              (-0.5,-\a,0)--(0,0,-1)
              (-0.5,\a,0)--(0,0,-1);
\fill[gray!30,opacity=0.5] (0,0,-1)--(0.5,\a,0)--(-0.5,\a,0)--(0, 0, 1)--(-0.5,-\a,0)--(0.5,-\a,0)--cycle;
\draw (0,0,-1)--(0.5,\a,0)--(-0.5,\a,0)--(0, 0, 1)--(-0.5,-\a,0)--(0.5,-\a,0)--cycle;   
\draw (-0.5,-\a,0)node[left]{$2$}--(0.5,-\a,0)node[below]{$3'$}--(1, 0, 0)node[below right]{$1$}--(0.5,\a,0)node[below]{$2'$}--(-0.5,\a,0)node[right]{$3$};     
\draw (0.5,-\a,0)--(0,0,1) (0,0,-1)--(1,0,0)--(0,0,1) (0.5,\a,0)--(0,0,1);
\end{tikzpicture}
\caption{On the left is the graph $\Gamma_6$, which is self-dual, and its Bier sphere, which is the $6$-cycle. On the right is a cone over $\Gamma_6$, which is also self-dual, and its Bier sphere, which is a suspension over $\Bier(\Gamma_6)$.}\label{fig:Biersuspension}
\end{figure}
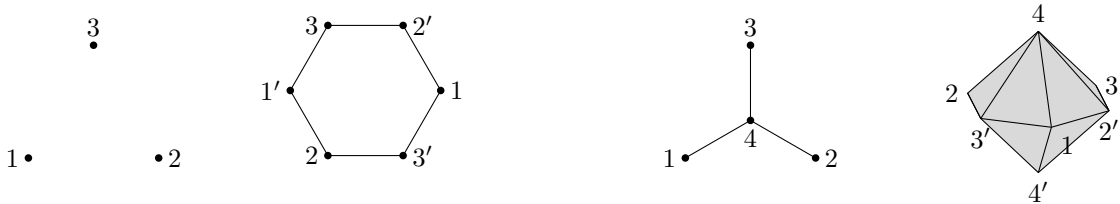

\begin{lemma}\label{BierAsSuspensionBasisLemma}
Let $K$ be a simplicial complex on $[m]$, $m\in V(K)$ and $L=\del_K(m)$. Then one has:
$$
K=\cone_{m}(L)\Longleftrightarrow \Bier(L)=\Bier(K)\cap 2^{[m-1]\sqcup [(m-1)']}.
$$ 
\end{lemma}
\begin{proof}
The implication $\Rightarrow$ follows directly from~\cite[Lemma 3.3 (c)]{LZ}:
$$
K=\cone_{m}(L)\Rightarrow \Bier(K)=\Sigma_{\{m,m'\}}(\Bier(L))\Rightarrow \Bier(L)=\Bier(K)\cap 2^{[m-1]\sqcup [(m-1)']}.
$$

Let us prove the implication $\Leftarrow$. Suppose $I\in\M(K)$ and $m\notin I$. Then $I\in\M(L)$, since $L$ is the deletion of $m$ from $K$. Therefore, $m'\in [m']\setminus I'\in\MF(\Bier(K))$ and hence 
$$
[(m-1)']\setminus I'\notin \MF(\Bier(K)_{[m-1]\sqcup [(m-1)']}),
$$
by the definition of a minimal non-face. On the other hand, by Alexander duality, we have:
$$
[(m-1)']\setminus I'\in \MF(\Bier(L)).
$$
Thus, $\Bier(L)$ and $\Bier(K)_{[m-1]\sqcup [(m-1)']}$ are different simplicial complexes on $[m-1]\sqcup [(m-1)']$, which contradicts our assumptions. This finishes the proof.
\end{proof}

\begin{lemma}\label{FewerRealVerticesCaseLemma}
Let $K$ be a simplicial complex on $[m]$. If either $f_0(K)$ or $f_0(K^\vee)$ is less than $m$, then 
$$
\chi(\Bier(K))=m-1
$$
if and only if one of the following combinatorial equivalences holds:
\begin{itemize}
\item $K=\Delta_{[m-1]}$;
\item $K=\Delta_{[m-2]}$;
\item $K=\Delta_{[m-1]}\cup \Delta_{[m-2]\cup\{m\}}$.
\end{itemize}
In all those cases, one has:
$$
\Bier(K)=\Bier(\Delta_{[m-3]}\ast \Gamma_4)=\Sigma^{m-3}(Z_4).
$$
\end{lemma}
\begin{proof}
Suppose $\chi(\Bier(K))=m-1$ and let $c\colon V(\Bier(K))\to C=\{c_1,c_2,\ldots,c_{m-1}\}$ be a proper coloring. Without loss of generality, we can assume $f_0(K)<m$ and $\{m\}\notin K$. Therefore, $[(m-1)']\in K^\vee$ and $C=\{c(1'),\ldots,c((m-1)')\}$. 
By \Cref{ChromaticNumberBasicEstimateProp}, it follows that $\chi(K^\vee)=m-1$.

If $\{m'\}\in K^\vee$, then $m'$ can not be linked by an edge to each of the vertices $1',2',\ldots,(m-1)'$ and moreover:
$$
c(i)=c(i')\text{ for all }1\leq i\leq f_0(K),
$$
since $\{i,j'\}\in\Bier(K)$ for each $1\leq i\leq f_0(K)$ and each $1'\leq j'\leq m'$ with $i\neq j$. For $j=m$ this implies that 
$$
c(m')\notin\{c(1'),\ldots,c((f_0(K))')\}. 
$$
Thus, we conclude that $f_0(K)\leq m-2$, and there exists $j\in\{f_0(K)+1,\ldots,m-1\}$ such that $\{j',m'\}\not\in K^\vee$, since $\chi(K^\vee)=m-1$. Furthermore, the vertex $j'$ must be unique in this case, since having more than one such $j'$ would imply the existence of two different $(m-3)$-dimensional simplices in $K$, which contradicts the condition $f_0(K)\leq m-2$. 

Thus, we can assume that $\{(m-1)',m'\}\notin K^\vee$ and $\{j',m'\}\in K^\vee$ for all $1\leq j\leq m-2$. Hence finally $f_0(K)=m-2$ and moreover $K=\Delta_{[m-2]}$, with two ghost vertices, $m-1$ and $m$. To complete the proper $(m-1)$-coloring $c$ of $\Bier(K)$ we set:
$$
c(m')=c((m-1)').
$$    

If $\{m'\}\notin K^\vee$, then $K^\vee=\Delta_{[(m-1)']}$ and hence $K=\Delta_{[m-1]}$. In this case, obviously, $\chi(\Bier(K))=m-1$ with the proper $(m-1)$-coloring $c$ given by: 
\[
i,i'\mapsto c_i\text{ for all }1\leq i\leq m-1.
\]
This final statement follows from the definition of a Bier sphere, completing the proof of the necessity implication.
The sufficiency implication follows from \Cref{ex:simplex}.
\end{proof}

\begin{lemma}\label{MinChromaticImpliesWeakSuspLemma}
Let $K$ be a simplicial complex on $[m]$ with $m\geq 4$ and $
f_0(K)=f_0(K^\vee)=m$. Then we have:
$$
\chi(\Bier(K))=m-1\Rightarrow \Bier(K)\text{ is a weak suspension}.
$$
\end{lemma}
\begin{proof}
Suppose $\Bier(K)$ is not a weak suspension and $\chi(\Bier(K))=m-1$. Let $c\colon V(\Bier(K))=[m]\sqcup [m']\to C$ be a proper coloring of $\Bier(K)$, where $|C|=m-1$.

By our assumption and \Cref{BierWeakSuspCriterionProp}, either $K$ or $K^\vee$ is not a weak cone. We can assume that $K$ is not a weak cone, and we may also assume that $\{1,2\}\notin K$. Then $\{3',4',\ldots,m'\}\in K^\vee$ and its vertices have pairwise different colors in $C$; let us define $C'=\{c(3'),\ldots, c(m')\}$. 

Since both $\{1,j'\},\{2,j'\}\in\Bier(K)$ for all $j\in\{3,\ldots,m\}$, it implies that 
$c(1),c(2)\notin C'$ and therefore $c(1)=c(2)=:\alpha$, since $|C'|=m-2$ and $|C|=m-1$. Thus, we have $C=C'\sqcup\{\alpha\}$.

Suppose either $1$ or $2$ is not an isolated vertex of $K$. Without loss of generality, we can assume that $\{1,3\}\in K$. It implies that $c(3)\in C'$. As $\{3,j'\}\in\Bier(K)$ for all $j\in\{4,\ldots,m\}$, we get that $c(3)=c(3')$.

Observe that $\{1,2'\},\{1',2\}\in\Bier(K)$ and hence $c(1'), c(2')\ne \alpha$.
It follows that there are vertices $i',j'\in\{3',4',\ldots,m'\}\in K^\vee$ such that 
$c(1')=c(i')$ and $c(2')=c(j')$. Hence, the corresponding 2-element sets are non-edges in $K^\vee$ and Alexander duality yields:
$$ 
[m]\setminus\{1,i\}, [m]\setminus\{2,j\}\in K.
$$

If $i\neq j$ the above inclusions imply that $\{2,j\},\{1,i\}\in K$, and so $c(i), c(j)\neq\alpha$. Therefore,
$$
c(i)=c(i')=c(1')\text{ and }c(j)=c(j')=c(2'),
$$
a contradiction, since $\{1',i\},\{2',j\}\in\Bier(K)$.

Thus, $i=j$ and $c(1')=c(i')=c(j')=c(2')=:\beta$ and 
$$ 
[m]\setminus\{1,i\}, [m]\setminus\{2,i\},\text{ and }[m]\setminus\{1,2\}\in K,
$$
since $\{1',i'\},\{2',i'\},\text{ and }\{1',2'\}\notin K^\vee$.

If $c(i)\neq\alpha$, then $c(i)\in C'$ and so $c(i)=c(i')$ by definition of a Bier sphere. Then $\{1',i\}\in \Bier(K)$ and $c(1')=c(i')=c(i)$, see above, a contradiction. Thus, $c(i)=\alpha$.

This implies that $\{1,i\},\{2,i\}\notin K$ and given $x$ such that $3\leq x\neq i\leq m$, we have: 
$$
x\in [m]\setminus\{1,i\}, [m]\setminus\{2,i\}, [m]\setminus\{1,2\}.
$$
It follows that $c(x)=c(x')$ for all $3\leq x\leq m, x\neq i$.

Finally, our coloring goes as follows:
\begin{align*}
c(1)&=c(2)=c(i)=\alpha\not\in C',\\
c(1')&=c(2')=c(i')=\beta\in C',\\
c(x)&=c(x')\text{ for }3\leq x\neq i\leq m.
\end{align*}

However, in such a Bier sphere, the pairs $\{j,j'\}$ for each $j\in [m]\setminus\{1,2,i\}$ can be taken as a weak suspension vertex pair and therefore $\Bier(K)$ is a weak suspension, which contradicts our initial assumption.

Now, suppose both $1$ and $2$ are isolated vertices of $K$. By Alexander duality, it follows that
$$
[m']\setminus\{1',i'\},[m']\setminus\{2',i'\}\in K^\vee\text{ for each }i\in\{3,\ldots, m\}.
$$

Recall that we also have $[m']\setminus\{1',2'\}\in K^\vee$.
Therefore, since $m\geq 4$ by the assumptions of the statement, we have 
$c(1'),c(2')\notin C'=C\setminus\{\alpha\}$, which implies that $c(1')=\alpha=c(2')$, a contradiction, as $\{1',2\},\{1,2'\}\in\Bier(K)$ and $\alpha=c(1)=c(2)$. Thus, $\Bier(K)$ is a weak suspension, which finishes the proof.
\end{proof}

\begin{lemma}\label{ChromaticNumberWeakSuspNotSuspLemma}
Let $K$ be a simplicial complex on $[m]$. Suppose $\Bier(K)$ is a weak suspension, which is not a suspension. Then we have:
$$
\chi(\Bier(K))=m.
$$
\end{lemma}
\begin{proof}
Due to \Cref{SuspensionStructureLemma}, we can assume that $\{1,1'\}$ is a weak suspension vertex pair for $\Bier(K)$. 
 
Since $\Bier(K)$ is not a suspension, there exists a simplex $\sigma\in \del_{K}(1)$ such that $\{1\}\cup\sigma\notin K$. Then, by definition of Alexander duality, its complement $[2',3',\ldots,m']\setminus\sigma'\in K^\vee$ and therefore, by definition of a Bier sphere, $\tau=([2',3',\ldots,m']\setminus\sigma')\cup\sigma$ is a maximal simplex in $\Bier(K)$.

By \Cref{BierWeakSuspCriterionProp}, $K$ is a weak cone with apex $1$. Hence the vertex $1$ is linked by an edge in $\Bier(K)$ to each vertex of $\sigma\in K$ and, therefore, to each vertex of $\tau$. Since $|\tau|=m-1$, it follows that the 1-skeleton of the full subcomplex of $\Bier(K)$ on the vertex set $\{1\}\cup\tau$ is a complete graph with $m$ vertices. This yields $\chi(\Bier(K))=m$, by \Cref{ChromaticNumberBasicEstimateProp}, which finishes the proof. 
\end{proof}

The next result characterizes the case when $\chi(\Bier(K))=m-1$ for a simplicial complex $K$ on $[m]$.

\begin{lemma}\label{SuspensionMinColorableLemma}
Let $K\neq\Delta_{[m]}$ be a simplicial complex on $[m]$ with $m\geq 3$.

Then $\chi(\Bier(K))=m-1$ if and only if $K$ is obtained from either $G_4$, or $\Gamma_4$, or $\Gamma_6$ in Figure~\ref{fig:1complexes} by iterative application of the following two simplicial operations:
\begin{itemize}
\item taking a cone;
\item taking Alexander dual. 
\end{itemize}
\end{lemma}
\begin{proof}
To prove the implication $\Leftarrow$, we observe that by~\cite[Lemma 3.3 (c)]{LZ} one has:
$$
\Bier(\Delta^{m-4}\ast \Gamma_4)=\Bier(\Delta^{m-4}\ast G_4)=\Sigma^{m-3}(Z_4)\text{ and }\Bier(\Delta^{m-4}\ast \Gamma_6)=\Sigma^{m-3}(Z_6).
$$
Coloring each time the vertex pair of the suspension into a new color, we get
$$
\chi(\Bier(K))\leq (m-3) + \chi(P_4)\text{ or }\chi(\Bier(K))\leq (m-3) + \chi(P_6),
$$
respectively. Both right hand sides are equal to $m-1$.
Application of \Cref{ChromaticNumberBasicEstimateProp} finishes the proof of this implication.

To prove the implication $\Rightarrow$, we argue by induction on $m$. For $m\leq 4$ the statement is true by the complete description of combinatorial Bier spheres in dimensions 1 and 2, see Example~\ref{OneDimBierExample} and~\cite[Theorem 2.16]{LS}. Suppose the statement holds for a fixed $m\geq 4$, and let us show it is true when $K$ is a simplicial complex on $[m+1]$ and $\chi(\Bier(K))=m$. Due to \Cref{FewerRealVerticesCaseLemma}, it remains to consider the case when both $K$ and $K^\vee$ have no ghost vertices: 
$$
f_0(K)=f_0(K^\vee)=m+1.
$$

Therefore, by 
\Cref{MinChromaticImpliesWeakSuspLemma}, $\Bier(K)$ is a weak suspension. Then, by \Cref{ChromaticNumberWeakSuspNotSuspLemma}, $\Bier(K)$ is a suspension.

Due to \Cref{SuspensionStructureLemma} we can assume that $\{m+1,(m+1)'\}$ is a vertex pair of $\Bier(K)$, hence, by~\Cref{BierSphereAsSuspensionLemma}, $m+1$ and $(m+1)'$ are the apexes of $K$ and $K^\vee$, respectively. Consider $R:=\del_{K}(m+1)$ as a simplicial complex on $[m]$. Then $K=\cone_{m+1}(R)$ and due to \Cref{BierAsSuspensionBasisLemma} one has: 
$$
\Bier(R)=\Bier(K)\cap 2^{[m]\cup [m']}.
$$

Since $\chi(\Bier(K))=m$ and $\Bier(K)$ is a suspension over $\Bier(R)$, by \Cref{ChromNumOfWeakSuspensionLemma}, we get:
$$
\chi(\Bier(R))=\chi(\Bier(K))-1=m-1.
$$

By the inductive assumption applied to $\Bier(R)$, we obtain that $R$ is obtained from either $G_4$, or $\Gamma_4$, or $\Gamma_6$ by iterative application of the following two simplicial operations: taking a cone and taking Alexander dual. 

It remains to recall that $K=\cone_{m+1}(R)$, which finishes the proof.
\end{proof}

In what follows, we use the notation $P_K$ for the combinatorial simple polytope such that 
$$
\Bier(K)=\partial P^*_K,
$$
when $\Bier(K)$ is a polytopal sphere. Now we are ready to introduce a complete description of chromatic numbers for Bier spheres, which summarizes the results of this section.

\begin{theorem}\label{BierChromaticTheorem}
Let $K\neq\Delta_{[m]}$ be a simplicial complex on $[m]$ with $m\geq 2$. Then 
\begin{enumerate}
\item[(a)] $m-1\leq \chi(\Bier(K))\leq m$;
\item[(b)] $\chi(\Bier(K))=m-1$ if and only if $\Bier(K)$ is polytopal and the next equivalent conditions hold:
\begin{itemize}
\item $K\in\{C^{m-3}(\Gamma_4)$, $C^{m-3}(\Gamma_4^{\vee})$, $C^{m-3}(G_4)$, $C^{m-3}(\Gamma_6)\}$;
\item $P_K\in\{I^{m-1}, I^{m-3}\times P_6\}$.
\end{itemize}
\end{enumerate}
\end{theorem}
\begin{proof}
The case $m=2$ is clear. In case $m\geq 3$, statement (a) follows directly from \Cref{ChromaticNumberBasicEstimateProp}, while statement (b) is a direct consequence of \Cref{SuspensionMinColorableLemma} and the fact that the operations of taking a cone of a simplicial complex and taking its Alexander dual commute. 
\end{proof}

\section{Buchstaber numbers and chordality of Bier spheres}

In this section, we discuss Buchstaber numbers of Bier spheres and their chordality. These combinatorial invariants and combinatorial property are closely related to certain important topological properties of polyhedral products studied in the framework of toric topology~\cite{BP02, TT}.

\begin{definition}
Let $K$ be a simplicial complex on $[m]$ with $m$ geometrical vertices. We call the {\emph{complex Buchstaber number}} of $K$ the maximal integer $r$ such that there exists a {\emph{characteristic map}}
$$
\Lambda\colon [m]\to \Z^{m-r}
$$
and we denote it by $s(K)$; that is, each simplex of $K$ is mapped to a part of a lattice basis of $\Z^{m-r}$. 

Similarly, given a prime number $p\in\N$, we call the {\emph{mod $p$ Buchstaber number}} of $K$ the maximal integer $r$ such that there exists a {\emph{mod $p$ characteristic map}}
$$
\Lambda_p\colon [m]\to \Z_p^{m-r}
$$
and we denote it by $s_{p}(K)$; that is, each simplex of $K$ is mapped to a linearly independent set of vectors in the $\Z_p$-vector space $\Z_p^{m-r}$.
\end{definition}

\begin{remark}
Note that this definition is equivalent to the one given in~\cite[Definition 5.3]{BVV}; moreover, in the case $p=2$, it is equivalent to the one of $s_\R(K)$ as the maximal rank of a real toric subgroup in $\Z_2^{m}$ acting freely on the real moment-angle-complex $\rk$ and similarly in the complex case, see~\cite{Ayz10,E09,Er14,FM}.
\end{remark}

It follows from the above definition that both $s(K)$ and $s_p(K)$, for any prime $p\in\N$, are combinatorial invariants of $K$.
Buchstaber invariant of a simplicial complex is known to be closely related to its chromatic number. First, given a simplicial complex $K$ with $f_0(K)=m$ and $\dim(K)=n-1$, the fundamental inequality $s(K)\leq m-n$ is well-known and can be found in~\cite{IZ01}. In the same paper, the Buchstaber and chromatic numbers were linked by the inequality $m-\chi(K)\leq s(K)$. Furthermore, for any simplicial complex, the real Buchstaber number bounds the complex Buchstaber number from above: $s(K)\leq s_2(K)$, see~\cite[Fact 15, p.5]{E09} as well as~\cite{Ayz10, FM, Er14}. Finally, for every prime $p\neq 2$, this is proved by literally the same argument, see~\cite{BVV}, where the mod $p$ Buchstaber invariant was introduced and studied, and therefore the chain of inequalities takes place:
$$
m-\chi(K)\leq s(K)\leq s_p(K)\leq m-n\text{ for any prime $p\in\N$}.
$$

Since $\dim(\Bier(K))=m-2$, applying these inequalities and Theorem~\ref{BierChromaticTheorem} to a Bier sphere, we get:
$$
f_0(\Bier(K))-m\leq s(\Bier(K))\leq s_p(\Bier(K))\leq f_0(\Bier(K))-(m-1)\text{ for any prime }p\in\N.
$$
Our next result shows that the last two inequalities above are, in fact, equalities for any Bier sphere. 

In~\cite[Theorem 3.4]{LS}, the real and complex Buchstaber invariants of $\Bier(K)$ were computed assuming that $K$ is considered alongside with its ghost vertices. Now, we are going to state and prove the general formulae for all Buchstaber numbers of $\Bier(K)$, when only geometrical vertices of $K$ are taken into account.

\begin{theorem}\label{BuchNumTheo}
Let $K\neq\Delta_{[m]}$ be a simplicial complex on $[m]$ with $m\geq 2$. Then we have
$$
s(\Bier(K))=s_p(\Bier(K))=f_0(K)-f_{m-2}(K)+1\text{ for any prime }p\in\N.
$$
\end{theorem}
\begin{proof}
Suppose $\Bier(K)$ is a polytopal sphere. 
Recall that by Theorem~\ref{BierChromaticTheorem} (a), for any $\Bier(K)$ with $K$ on $[m], m\geq 2$ one has either $\chi(\Bier(K))=m-1$, or $\chi(\Bier(K))=m$. Note that when $\chi(\Bier(K))=m-1=\dim(\Bier(K))+1$, the above inequalities immediately imply that the complex and all mod $p$ Buchstaber numbers of a Bier sphere are maximal possible in the first case. When $\chi(\Bier(K))=m=\dim(\Bier(K))+2$, it follows from~\cite[Theorem (7)]{Er08} that the chain of equalities and inequalities
\begin{multline*}
s(\Bier(K))\geq f_0(\Bier(K))-\chi(\Bier(K))+s(\Delta^{\chi(\Bier(K))-1}_{\dim(\Bier(K))})=\\
=f_0(\Bier(K))-m+s(\Delta^{m-1}_{m-2})\geq f_0(\Bier(K))-m+1
\end{multline*}
holds, where $\Delta_{p}^{q}$ denotes the $p$-skeleton of the $q$-dimensional simplex. This shows that the complex and all mod $p$ Buchstaber numbers of a Bier sphere are also maximal possible in the second case. Now we consider the case of a general $\Bier(K)$ and we are going to introduce a canonical characteristic map for every Bier sphere.
First, let us compute the number of geometrical vertices of $\Bier(K)$ in terms of the $f$-vector of $K$. By~\cite[Corollary 2.7]{LS}, the Bier sphere $\Bier(K)$ has $(m-f_0(K))+f_{m-2}(K)$ ghost vertices, which is just the sum of the numbers of ghost vertices in $K$ and in its Alexander dual complex $K^\vee$. Hence, the number of geometrical vertices of the Bier sphere equals 
$$
|V(\Bier(K))|=f_0(\Bier(K))=2m-(m-f_0(K)+f_{m-2}(K))=m+f_0(K)-f_{m-2}(K).
$$

Denote by $\{e_1,e_2,\ldots,e_{m-1}\}$ the standard basis of $\Z^{m-1}$ and consider the map 
$$
\phi\colon [m]\sqcup [m']\to \Z^{m-1}\text{ such that }\phi\colon i,i'\mapsto e_{i},\text{ for all }1\leq i\leq m-1,\text{ and }
\phi\colon m,m'\mapsto e_1+e_2+\ldots+e_{m-1}.
$$

Let us define $\Lambda_{(p)}$ to be the restriction of $\phi$ to $V(\Bier(K))$. Then $\Lambda_{(p)}$ is a characteristic map for $\Bier(K)$, both in the complex and mod $p$ case for each prime $p\in\N$. This follows from the fact that maximal simplices of $\Bier(K)$ have cardinality $(m-1)$ and do not contain a pair $\{i,i'\}$ for any $1\leq i\leq m$.

Set $r=f_0(\Bier(K))-(m-1)$. By the definition of a Buchstaber number, this shows that
$$
s_p(K)\geq s(K)\geq r=f_0(\Bier(K))-(m-1)\text{ for any prime }p\in\N.
$$

On the other hand, we always have:
$$
s(K)\leq s_p(K)\leq f_0(\Bier(K))-(\dim(\Bier(K))+1)=f_0(\Bier(K))-(m-1)\text{ for any prime }p\in\N.
$$

Thus, one gets the following equality:
$$
s(K)=s_p(K)=f_0(\Bier(K))-(m-1)\text{ for any prime }p\in\N. 
$$

Finally, the right hand side equals $(m+f_0(K)-f_{m-2}(K))-(m-1)=f_0(K)-f_{m-2}(K)+1$.
\end{proof}


Now we turn to the discussion of chordality for Bier spheres.

\begin{definition}
We say that a graph $\Gamma$ is \emph{chordal} if there are no induced cycles of length greater than $3$ in $\Gamma$; in other words, if $\Gamma_I=Z_k$, then $k=3$; such a full subcomplex is called \emph{chordless}. Finally, we call a simplicial complex $K$ \emph{chordal} if $\sk^1(K)$ is a chordal graph.
\end{definition}

\begin{definition}
We write $P\in\vc^{k}(\Delta^n)$ with $n\geq 2, k\geq 0$ and say that $P$ is a \emph{truncation polytope} if $P$ is combinatorially equivalent to a convex simple $n$-dimensional polytope obtained from $\Delta^n$ by a sequence of $k$ vertex truncations by means of hyperplanes in general position.  
\end{definition}

\begin{figure}[ht]
\tdplotsetmaincoords{70}{80}
\begin{tikzpicture}[tdplot_main_coords,line join=round,scale=1.2]
\pgfmathsetmacro\a{sqrt(2/9)}
\pgfmathsetmacro\b{sqrt(2/3)}
\pgfmathsetmacro\c{1/3}
\pgfmathsetmacro\r{0.3}
\path 
    coordinate(A) at (2*\a, 0, -\c)
    coordinate(B) at (-\a, \b, -\c)
    coordinate(C) at (-\a, -\b, -\c)
    coordinate(D) at (0,0,1)
     coordinate(AA) at ($(D)!\r!(A)$)
    coordinate(BB) at ($(D)!\r!(B)$)
    coordinate(CC) at ($(D)!\r!(C)$)
    coordinate(DB) at ($(B)!\r!(D)$)
    coordinate(AB) at ($(B)!\r!(A)$)
    coordinate(CB) at ($(B)!\r!(C)$)
    coordinate(DC) at ($(C)!\r!(D)$)
    coordinate(AC) at ($(C)!\r!(A)$)
    coordinate(BC) at ($(C)!\r!(B)$)
    coordinate(DA) at ($(A)!\r!(D)$)
    coordinate(BA) at ($(A)!\r!(B)$)
    coordinate(CA) at ($(A)!\r!(C)$);
\clip(-1,-0.8,-1) rectangle (1,0.6,1.4);
\draw[dashed] (-\a, \b, -\c)--(-\a, -\b, -\c);
\fill[gray!30,opacity=0.5] (2*\a, 0, -\c)--(-\a, \b, -\c)--(0,0,1)--(-\a, -\b, -\c)--cycle;
\draw (2*\a, 0, -\c)--(0,0,1)--(-\a, -\b, -\c)--cycle;
\draw (2*\a, 0, -\c)--(-\a, \b, -\c)--(0,0,1)--cycle;
\end{tikzpicture}
\hskip5mm
\begin{tikzpicture}[tdplot_main_coords,line join=round,scale=1.2]
\clip(-1,-0.8,-1) rectangle (1,0.6,1.4);
\draw[dashed] (B)--(C);
\fill[gray!30,opacity=0.5] (A)--(B)--(BB)--(CC)--(C)--cycle;
\draw (A)--(AA)
      (B)--(BB)
      (C)--(CC)
      (B)--(A)--(C)
      (BB)--(AA)--(CC)--cycle;
\end{tikzpicture}
\hskip5mm
\begin{tikzpicture}[tdplot_main_coords,line join=round,scale=1.2]
\clip(-1,-0.8,-1) rectangle (1,0.6,1.4);
\draw[dashed] (AB)-- (CB) (C)--(CB)--(DB);
\fill[gray!30,opacity=0.5] (A)--(AB)--(DB)--(BB)--(CC)--(C)--cycle;
\draw (A)--(AA)
      (DB)--(BB)
      (C)--(CC)
      (DB)--(AB)--(A)--(C)
      (BB)--(AA)--(CC)--cycle;
\end{tikzpicture}
\hskip5mm
\begin{tikzpicture}[tdplot_main_coords,line join=round,scale=1.2]
\clip(-1,-0.8,-1) rectangle (1,0.6,1.4);
\draw[dashed] (AB)-- (CB) (AC)--(BC) (DC)--(BC)--(CB)--(DB);
\fill[gray!30,opacity=0.5] (A)--(AB)--(DB)--(BB)--(CC)--(DC)--(AC)--cycle;
\draw (A)--(AA) (BB)--(CC)
      (BB)--(DB)--(AB)--(A)--(AC)--(DC)--(CC)--(AA)--cycle;
\end{tikzpicture}
\hskip5mm
\begin{tikzpicture}[tdplot_main_coords,line join=round,scale=1.2]
\clip(-1,-0.8,-1) rectangle (1,0.6,1.4);
\draw[dashed] (AB)-- (CB) (AC)--(BC) (DC)--(BC)--(CB)--(DB);
\fill[gray!30,opacity=0.5] (CA)--(BA)--(AB)--(DB)--(BB)--(CC)--(DC)--(AC)--cycle;
\draw (BA)--(DA)--(AA) (CA)--(DA) (BB)--(CC)
      (BB)--(DB)--(AB)--(BA)--(CA)--(AC)--(DC)--(CC)--(AA)--cycle;
\end{tikzpicture}
\caption{Some examples of types $\vc^0(\Delta^3)$, $\vc^1(\Delta^3)$, $\vc^2(\Delta^3)$, $\vc^3(\Delta^3)$, and $\vc^4(\Delta^3)$.}\label{fig:truncated}
\end{figure}
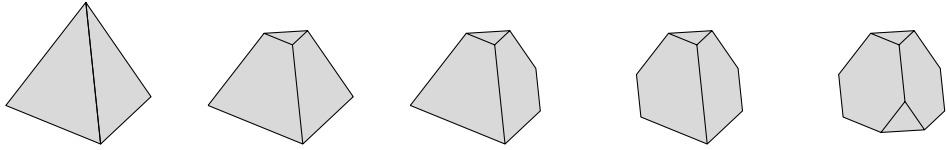

\begin{remark}
Note that the polytope $Q$ combinatorially dual to a truncation polytope $P$ is combinatorially equivalent to a convex simplicial $n$-dimensional polytope obtained from $\Delta^n$ by a sequence of $k$ stellar subdivisions in facets. Such a polytope is called a \emph{stacked polytope}. Observe that the combinatorial types of $P$ and $Q$ depend on the vertices truncated and facets subdivided, respectively, when $n\geq 3$ and $k\geq 3$. Namely, there is only one combinatorial type of $\vc^0(\Delta^3)$, $\vc^1(\Delta^3)$, and $\vc^2(\Delta^3)$, whereas there are 3 combinatorially different realizations of $\vc^3(\Delta^3)$ (see \Cref{fig:types}) and 7 combinatorially different realizations of $\vc^4(\Delta^3)$ \cite[Table 8]{HRS}.
\end{remark}

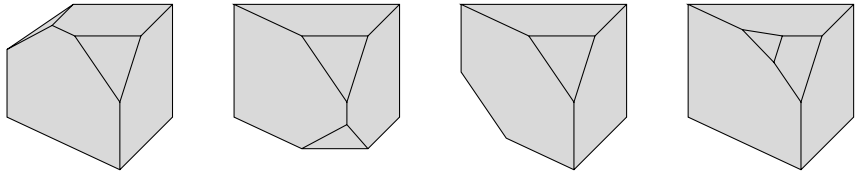
\begin{figure}[ht]
\begin{tikzpicture}
\pgfmathsetmacro\r{0.4}
\path 
    coordinate(A) at (0,0)
    coordinate(B) at (0.7,0.7)
    coordinate(C) at (-1.5,0.7)
    coordinate(A1) at ($(A)+(0,1.5)$)
    coordinate(B1) at ($(B)+(0,1.5)$)
    coordinate(C1) at ($(C)+(0,1.5)$)
    coordinate(AA) at ($(A1)!\r!(A)$)
    coordinate(AB) at ($(A1)!\r!(B1)$)
    coordinate(AC) at ($(A1)!\r!(C1)$)
    coordinate(CA) at ($(C1)!\r!(A1)$)
    coordinate(CB) at ($(C1)!\r!(B1)$)
    coordinate(CC) at ($(C1)!\r!(C)$)
    coordinate(XA) at ($(A)!\r!(A1)$)
    coordinate(XB) at ($(A)!\r!(B)$)
    coordinate(XC) at ($(A)!\r!(C)$)
    coordinate(YA) at ($(XA)!\r!(A)$)
    coordinate(YB) at ($(AB)!\r!(XA)$)
    coordinate(YC) at ($(AC)!\r!(XA)$)
    coordinate(ZA) at ($(C)!\r!(A)$)
    coordinate(ZB) at ($(C)!\r!(B)$)
    coordinate(ZC) at ($(C)!\r!(C1)$)
    coordinate(WA) at ($(AC)!\r!(AA)$)
    coordinate(WB) at ($(AC)!\r!(AB)$)
    coordinate(WC) at ($(AC)!0.2!(C1)$);
\clip (-1.6,-0.1) rectangle (0.8,2.3);    
\draw[dashed] (B)--(C);
\fill[gray!30,opacity=0.5] (A)--(B)--(B1)--(CB)--(CC)--(C)--cycle;
\draw (CB)--(CC)--(C)--(A)--(B)--(B1)--(CB)--(CA)--(AC)--(AB)--(AA)--(AC) (AB)--(B1) (AA)--(A) (CC)--(CA);
\end{tikzpicture}
\hskip5mm
\begin{tikzpicture}
\clip (-1.6,-0.1) rectangle (0.8,2.3);    
\draw[dashed] (B)--(C);
\fill[gray!30,opacity=0.5] (XC)--(XB)--(B)--(B1)--(C1)--(C)--cycle;
\draw (C1)--(C)--(XC)--(XB)--(B)--(B1)--(C1)--(AC)--(AB)--(AA)--(AC) (AB)--(B1) (AA)--(XA)--(XB) (XC)--(XA);
\end{tikzpicture}
\hskip5mm
\begin{tikzpicture}
\clip (-1.6,-0.1) rectangle (0.8,2.3);    
\draw[dashed] (B)--(ZB)--(ZC) (ZB)--(ZA);
\fill[gray!30,opacity=0.5] (A)--(B)--(B1)--(C1)--(ZC)--(ZA)--cycle;
\draw (C1)--(ZC)--(ZA)--(A)--(B)--(B1)--(C1)--(AC)--(AB)--(AA)--(AC) (AB)--(B1) (A)--(AA);
\end{tikzpicture}
\hskip5mm
\begin{tikzpicture}
\clip (-1.6,-0.1) rectangle (0.8,2.3);    
\draw[dashed] (B)--(C);
\fill[gray!30,opacity=0.5] (A)--(B)--(B1)--(C1)--(C)--cycle;
\draw (C1)--(C)--(A)--(B)--(B1)--(C1)--(WC)--(WA)--(WB)--(AB)--(AA)--(WA) (AB)--(B1) (A)--(AA) (WC)--(WB);
\end{tikzpicture}
\caption{Examples of types of $\vc^3(\Delta^3)$. The last two polytopes are combinatorially equivalent.}\label{fig:types}
\end{figure}
 
The next result follows from \cite[Theorem 8.5]{Kalai}. We give it with proof here for the sake of completeness. 

\begin{proposition}\label{TruncationAreChordalProp}
Let $P\in\vc^{k}(\Delta^n)$ with $n\geq 3, k\geq 0$ be a truncation polytope. Then $K_P$ is chordal. 
\end{proposition}
\begin{proof}
Since for $P=\Delta^n, n\geq 3$ the statement is true, we can fix $n$ and use induction on $k\geq 0$.

To make the induction step, suppose $P\in\vc^{k+1}(\Delta^n)$ and $n\geq 3, k\geq 0$. Then the dual stacked polytope $Q$ was obtained by performing a stellar subdivision with apex $v$ in a facet $\sigma=\{i_1,\ldots,i_n\}$ of a stacked polytope $Q'$ which is dual to a certain simple polytope in $\vc^k(\Delta^n)$.

Suppose $S$ was an induced $p$-cycle in the graph $\sk^1(\partial Q)$ and $p\geq 4$. We have two possible cases. 

If $v\notin S$, then $S$ was an induced cycle in $\partial Q'$, which is chordal by inductive assumption, hence $p=3$, a contradiction. Indeed, since $\sigma$ was a simplex in $\partial Q'$, the induced chordless cycle $S$ of length $p\geq 4$ could intersect it in no more than two vertices. Then $S$ remains a full subcomplex in $\partial Q$ after the stellar subdivision in $\sigma$ was performed. 

If $v\in S$, then at least two vertices, $u$ and $w$, of $\sigma$ must belong to the cycle $S$. Therefore, $S$ must contain all the three edges, $\{u,v\}, \{u,w\}$, and $\{v,w\}$. This shows that $S$ is not a chordless cycle, a contradiction. This finishes the proof.
\end{proof}

The converse of the previous statement is not true; e.g. if $P$ was obtained from $\Delta^2\times\Delta^2$ by a vertex cut, then $P$ is not a truncation polytope, but $K_P$ is a chordal complex. We are ready to state the result, which provides the classification of all chordal Bier spheres in dimensions greater than one. It implies that for Bier spheres of dimension greater than one, being a stacked sphere and being a chordal complex are equivalent.

\begin{lemma}\label{ChordalKeyLemma}
Let $K\neq\Delta_{[m]}$ be a simplicial complex on $[m]$ with $m\geq 4$. Then $\Bier(K)$ is chordal if and only if either $K$ or $K^\vee$ contains no edges.
\end{lemma}
\begin{proof}
The implication $\Leftarrow$ follows immediately from \Cref{TruncationAreChordalProp}. Indeed, the case of a void complex is obvious. Suppose $k>0$ and $K^\vee$ is generated by its vertex set $\{1',2',\ldots,k'\}$. It follows from the definition of a Bier sphere that $\Bier(K)$ is obtained from $\partial\Delta_{[m]}$ by stellar subdivisions of the facets $[m]\setminus\{i\}$ by vertices $i'$ for $1\leq i\leq k$. The corresponding simple polytope $P_K$ is then a truncation polytope obtained from $\Delta^{m-1}$ by cutting off $k$ of its $m$ vertices, see~\Cref{fig:truncated} for the case $m=4$ and $0\leq k\leq 4$. Therefore, we obtain a polytopal Bier sphere, whose Bier polytope $P_K$ is of the type in $\vc^{k}(\Delta^{m-1})$.

To prove the implication $\Rightarrow$, suppose that $\Bier(K)$ is chordal and each of the graphs, $\sk^1(K)$ and $\sk^1(K^\vee)$, contains an edge. Without loss of generality, we can assume that $\{1,2\}\in K$. 

Then $\{1',2'\}\notin K^\vee$, otherwise $\Bier(K)$ contains a chordless 4-cycle on the vertices $\{1,2,1',2'\}$ as a full subcomplex and therefore $\Bier(K)$ is non-chordal, a contradiction. Note that if $\{i\}\notin K$ for some $i\geq 3$, then Alexander duality would imply that $\{1',2'\}\subset [m']\setminus\{i'\}\in K^\vee$, which contradicts $\{1',2'\}\notin K^\vee$. It follows that $f_0(K)=m$ and $K$ has no ghost vertices. The same argument shows that $K^\vee$ also has no ghost vertices.

Now, we are going to show that both $\{1,i\}$ and $\{2,i\}$ are in $K$ for each $i$ with $3\leq i\leq m$. Indeed, otherwise one of the following three cases occurs. If there exists an $i, 3\leq i\leq m$ such that $\{1,i\}, \{2,i\}\notin K$, then $\Bier(K)$ contains a chordless 5-cycle on the vertices $\{1,2,1',i,2'\}$ as a full subcomplex and therefore $\Bier(K)$ is non-chordal, a contradiction. 
If there exists an $i, 3\leq i\leq m$ such that $\{1,i\}\notin K$ and $\{2,i\}\in K$, then $\Bier(K)$ contains a chordless 4-cycle on the vertices $\{1,2,i,2'\}$ as a full subcomplex and therefore $\Bier(K)$ is non-chordal, a contradiction.
Finally, if there exists an $i, 3\leq i\leq m$ such that $\{1,i\}\in K$ and $\{2,i\}\notin K$, then $\Bier(K)$ contains a chordless 4-cycle on the vertices $\{1,2,1',i\}$ as a full subcomplex and therefore $\Bier(K)$ is non-chordal, a contradiction.

Next, observe that $\{i,j\}\in K$ for every pair of vertices $i, j$ such that $3\leq i\neq j\leq m$. Indeed, if $\{i,j\}\notin K$, then $\Bier(K)$ contains a chordless 4-cycle on the vertices $\{1,i,1',j\}$ as a full subcomplex and therefore $\Bier(K)$ is non-chordal, a contradiction.

Thus, $\sk^1(K)$ is a complete graph with $m$ vertices. Suppose there was an edge $\{i',j'\}\in K^\vee$ with $1\leq i\neq j\leq m$. Then $\Bier(K)$ contains a chordless 4-cycle on the vertices $\{i,j,i',j'\}$ as a full subcomplex and therefore $\Bier(K)$ is non-chordal, a contradiction. It implies that $\sk^1(K^\vee)$ has no edges, which contradicts our initial assumption. 
This finishes the proof.
\end{proof}

We summarize our results on chordality for Bier spheres in the next statement.

\begin{theorem}\label{ChordalBierTheorem}
Let $K\neq\Delta_{[m]}$ be a simplicial complex on $[m]$ with $m\geq 2$. Then the following three statements are equivalent:
\begin{itemize}
\item[(a)] $\Bier(K)$ is a chordal complex;
\item[(b)] $\Bier(K)$ is a polytopal sphere and one of the next three cases holds:
\begin{enumerate}
\item $P_K=\Delta^1$, where $m=2$;
\item $P_K=\Delta^2$, where $m=3$;
\item $P_K\in\vc^{k}(\Delta^{m-1})$, where $0\leq k\leq m$ and $m\geq 4$. \end{enumerate}
\item[(c)] One of the next three cases holds:
\begin{enumerate}
\item $m=2$;
\item $m=3$ and either $K$ or $K^\vee$ equals $\partial\Delta^2$;
\item $m\geq 4$ and either $K$ or $K^\vee$ has no edges.
\end{enumerate}
\end{itemize}
If $m\leq 3$, then $\Bier(K)$ is a stacked sphere being either the boundary of a simplex, or the $k$-cycle with $k=4, 5, $ or $6$. If $m\geq 4$, then $P_K$ is obtained by cutting
off $k$ vertices of the simplex $\Delta^{m-1}$, where $k$ is the number of vertices of the corresponding complex without edges; in this case, {\rm (a)}, {\rm (b)} and {\rm (c)} are equivalent to:
\begin{itemize}
\item[(d)] $\Bier(K)$ is a stacked sphere.
\end{itemize}
\end{theorem}
\begin{proof}
Let us prove the equivalence of (a) and (c). In case $m=2$ the unique $0$-dimensional Bier sphere is a chordal complex. In case $m=3$ the only chordal $1$-dimensional Bier sphere is the boundary of a triangle, see Figure~\ref{fig:1BierSpehers}. If $m\geq 4$, then the application of Lemma~\ref{ChordalKeyLemma} finishes the proof of the equivalence (a) $\Leftrightarrow$ (c). The equivalence (b) $\Leftrightarrow$ (c) follows directly from the definition of Bier sphere. Finally,  if $m\leq 3$, then $\Bier(K)$ is a stacked sphere due to Example~\ref{OneDimBierExample}; for $m\geq 4$, the implication (b) $\Rightarrow$ (d) is obvious, while the implication (d) $\Rightarrow$ (a) follows from Proposition~\ref{TruncationAreChordalProp}.     
\end{proof}


\subsection*{Acknowledgements}
The authors are grateful to Djordje Barali\'c, Matvey Sergeev, and Rade \v{Z}ivaljevi\'c for numerous fruitful discussions, valuable comments and suggestions. 

Limonchenko was supported by the Serbian Ministry of Science, Technological Development, and Innovation through the Mathematical Institute of the Serbian Academy of Sciences and Arts.

Vavpeti\v c was supported by the Slovenian Research and Innovation Agency program P1-0292 and the grant J1-4031.

\normalsize

\end{document}